\documentclass[12pt,oneside,english]{amsart}
\usepackage[T1]{fontenc}
\usepackage[latin9]{inputenc}
\usepackage{amsthm}
\usepackage{amstext}
\usepackage{amssymb}
\usepackage{esint}

\makeatletter
\numberwithin{equation}{section}
\numberwithin{figure}{section}
\theoremstyle{plain}
\newtheorem{thm}{\protect\theoremname}
  \theoremstyle{definition}
  \newtheorem{defn}[thm]{\protect\definitionname}
  \theoremstyle{remark}
  \newtheorem{rem}[thm]{\protect\remarkname}
  \theoremstyle{plain}
  \newtheorem{lem}[thm]{\protect\lemmaname}
  \theoremstyle{plain}
  \newtheorem{prop}[thm]{\protect\propositionname}
  \theoremstyle{plain}
  \newtheorem{cor}[thm]{\protect\corollaryname}

\makeatother

\usepackage{babel}
  \providecommand{\corollaryname}{Corollary}
  \providecommand{\definitionname}{Definition}
  \providecommand{\lemmaname}{Lemma}
  \providecommand{\propositionname}{Proposition}
  \providecommand{\remarkname}{Remark}
\providecommand{\theoremname}{Theorem}

\begin{document}

\title{Divided Differences \& Restriction Operator on Paley-Wiener Spaces
$PW_{\tau}^{p}$ for $N-$Carleson Sequences}

\author{Fr\'ed\'eric Gaunard}


\begin{abstract}
For a sequence of complex numbers $\Lambda$
we consider the restriction operator $R_{\Lambda}$ defined on Paley-Wiener
spaces $PW_{\tau}^{p}$ ($1<p<\infty$). Lyubarskii and Seip gave necessary and sufficient
conditions on $\Lambda$ for $R_{\Lambda}$ to be an isomorphism between
$PW_{\tau}^{p}$ and a certain weighted $l^{p}$ space. The Carleson condition
appears to be necessary. We extend their result to $N-$Carleson sequences
(finite unions of $N$ disjoint Carleson sequences). More precisely, we
give necessary and sufficient conditions for $R_{\Lambda}$ to be
an isomorphism between $PW_{\tau}^{p}$ and an appropriate sequence space
involving divided differences. 
\end{abstract}

\date{\today}

\maketitle

\section{Introduction}

Let $X$ be a Banach space of analytic functions defined on a domain
$\Omega$ of the complex plane and $\Lambda$ a sequence of points lying in $\Omega$.
The \emph{restriction operator} $R_{\Lambda}$ associated to $\Lambda$
is defined on $X$ by 
\[
R_{\Lambda}:X\ni f\mapsto\left(f\left(\lambda\right)\right)_{\lambda\in\Lambda}\in\mathbb{C}^{\Lambda}.
\]
Our aim is to describe the range of $R_{\Lambda}$, denoted
by $X\vert\Lambda$, as well as the injectivity of $R_{\Lambda}$.
This problem is related to interpolation problems in $X$
and to geometrical properties of reproducing kernels in $X^{\star}$.
See \cite{HNP81}, \cite[Part D]{Ni02b} or \cite{Se04}.

\bigskip

In the late 1950s and early 1960s, Carleson \cite{Ca58} ($p=\infty$) and Shapiro
and Shields \cite{SS61} ($1\leq p<\infty$) showed that $R_{\Lambda}$
is surjective from the Hardy space onto a suitable weighted $l^{p}$ space
if and only if $\Lambda$ satisfies a certain separation condition,
the so-called \emph{Carleson} condition (more precise definitions
below). Notice that, in Hardy spaces, as soon as the sequence satisfies
the Blaschke condition, $R_{\Lambda}$ cannot be injective.

The results of Carleson and Shapiro-Shields have been generalized to finite unions
of Carleson sequences (which are called $N-$\emph{Carleson} sequences)
by Vasyunin \cite{Va84} ($p=\infty$) and Hartmann \cite{Ha96b}
($1<p<\infty$). A similar result has
been obtained by Bruna, Nicolau and \O yma \cite{BNO96}.
In this more general situation the description of the range of $R_{\Lambda}$ 
involves divided differences. 

Many authors like Hrushev, Nikolskii, Pavlov \cite{HNP81} or Minkin
\cite{Mi92}, have investigated interpolation problems in Paley-Wiener
spaces using tools from operator theory (for instance
invertibility criteria for a suitable Toeplitz operator)
since the 1970s. Note that these spaces 
can be considered as special cases of backward shift invariant 
subspaces in Hardy spaces.
More recently,
Lyubarskii and Seip \cite{LS97} have characterized the sequences $\Lambda$
for which the associated restriction operator is an isomorphism between
the Paley-Wiener space and an appropriate weighted $l^{p}$ space. 
Their proof is in a sense more elementary and allows to consider sequences
defined on the whole complex plane while the methods of
Hrushev, Nikolskii, Pavlov intrinsically restrict the problem to sequences in
a half-plane. 

Here we investigate a generalization of Lyubarskii and Seip's result to $N-$Carleson sequences, in the spirit of Hartmann. Observe first that the Carleson condition turns 
out to be necessary for the classical interpolation problem in the Paley-Wiener space.
Now, starting from an $N-$Carleson sequence
$\Lambda$, we want to find necessary and sufficient conditions on
$\Lambda$ for $R_{\Lambda}$ to be an isomorphism between the Paley-Wiener
space and an appropriate sequence space involving now divided differences.

\bigskip

Let us  fix the notation and the results we mentioned above. We first
recall the definition of the \emph{Hardy space, for $1\leq p<\infty$,}
\[
H^{p}\left(\mathbb{C}_{a}^{\pm}\right):=\left\{ f\in\text{Hol}\left(\mathbb{C}_{a}^{\pm}\right):\:\sup_{y\gtrless a}\int_{\mathbb{R}}\left|f\left(x+iy\right)\right|^{p}dx<\infty\right\} 
\]
on the half-plane
\[
\mathbb{C}_{a}^{\pm}:=\left\{ z\in\mathbb{C}:\;\text{Im}\left(z\right)\gtrless a\right\} ,\quad\left(a\in\mathbb{R}\right).
\]
For $p=\infty$, 
\[
H^{\infty}\left(\mathbb{C}_{a}^{\pm}\right):=\left\{ f\in\text{Hol}\left(\mathbb{C}_{a}^{\pm}\right):\:\sup_{z\in\mathbb{C}_{a}^{\pm}}\left|f\left(z\right)\right|<\infty\right\} .
\]
For short
we will write $\mathbb{C}^{\pm}:=\mathbb{C}_{0}^{\pm}$ and $H_{\pm}^{p}:=H^{p}\left(\mathbb{C}^{\pm}\right)$.
A function $I\in H^{\infty}\left(\mathbb{C}_{a}^{\pm}\right)$ 
satisfying $\left|I\left(x+ia\right)\right|=1$
a.e. $x\in\mathbb{R}$ is called an \emph{inner function.}

\bigskip

As previously mentioned, Carleson \cite{Ca58}, Shapiro and Shields
\cite{SS61} solved the interpolation problem in the Hardy space. Their
results were obtained in the unit disk, but translate clearly 
to any half-plane. Setting
\[
l^{p}\left(\left|\text{Im}\left(\lambda_{n}\right)-a\right|\right):=\left\{ u=\left(u_{n}\right)_{n\geq1}:\;\sum_{n\geq1}\left|\text{Im}\left(\lambda_{n}\right)-a\right|\left|u_{n}\right|^{p}<\infty\right\} ,
\]
 we can state their result as follows. If $\Lambda=\left\{ \lambda_{n}:\: n\geq1\right\} \subset\mathbb{C}_{a}^{\pm}$,
then
\[
H^{p}\left(\mathbb{C}_{a}^{\pm}\right)\vert\Lambda=l^{p}\left(\left|\text{Im}\left(\lambda_{n}\right)-a\right|\right)
\]
if and only if $\Lambda$ satisfies the \emph{Carleson} condition
\begin{equation}
\inf_{\lambda\in\Lambda}\prod_{\underset{\mu\not=\lambda}{\mu\in\Lambda}}\left|\frac{\lambda-\mu}{\lambda-\overline{\mu}-2ia}\right|>0.\label{eq:Carleson}
\end{equation}
Such sequences will be simply called \emph{Carleson sequences}.

We consider now the \emph{Paley-Wiener space} $PW_{\tau}^{p}$ (for
$1\leq p<\infty$) which consists of all entire functions of exponential
type at most $\tau$ satisfying
\[
\left\Vert f\right\Vert _{p}^{p}=\int_{\mathbb{R}}\left|f\left(x\right)\right|^{p}dx<\infty.
\]
It is well-known (see e.g. \cite{Le96}) that in the case $p=2$,
the Fourier transform is an isometric isomorphism between $PW_{\tau}^{2}$
and $L^{2}\left(-\tau,\tau\right)$ which allows to reformulate
the problem in terms of geometrical properties of exponentials
in $L^{2}$ (we still refer to \cite{HNP81}). From the Plancherel-Poly\`a
inequality (see Proposition \ref{plancherel-polya} below), it follows
that $PW_{\tau}^{p}=e^{-i\tau\cdot}K_{I^{\tau}}^{p}$,
where 
\[
K_{I^{\tau}}^{p}:=H_{+}^{p}\cap\overline{I^{\tau}}H_{-}^{p}
\]
is the backward shift invariant subspace associated with the inner function
$I^{\tau}\left(z\right):=\exp\left(2i\tau z\right)$, $z\in\mathbb{C}^{+}$.
In particular, the Paley-Wiener space can be considered as a subspace of
the Hardy space.

\bigskip

Luybarskii and Seip \cite{LS97} gave necessary and sufficient conditions
for $R_{\Lambda}$ to be an isomorphism from $PW_{\tau}^{p}$ onto
the weighted sequence space 
$l^{p}\left(e^{-p\tau|\text{Im}(\lambda_{n})|}\left(1+|\text{Im}(\lambda_{n})|\right)\right)$.
Their proof is based on the boundedness of the Hilbert transform in
certain weighted Hardy space. 

Recall that the \emph{Hilbert transform} $\mathcal{H}$ is defined by
\begin{equation}
\mathcal{H}f(z)=\int_{-\infty}^{+\infty}\frac{f(t)}{t-z}dt,\label{eq:Hilbert transform def}
\end{equation}
where the integral has to be understood as a principle value integral for real $z$. 
It is known (see e.g \cite{HMW73} and \cite{Gar81})
that, if $w>0$, $\mathcal{H}$ is bounded from the weighted space
\[
L^{p}(w):=\left\{ f\text{ meas. on }\mathbb{R}:\;\int_{\mathbb{R}}\left|f\right|^{p}wdm<\infty\right\} 
\]
into itself, if and only if $w$ satisfies the \emph{Muckenhoupt}
$(A_{p})$ condition
\[
(A_{p})\qquad\qquad\sup_{I}\left(\frac{1}{|I|}\int_{I}w\right)\left(\frac{1}{|I|}\int_{I}w^{-\frac{1}{p-1}}\right)^{\frac{p}{p-1}}<\infty,
\]
where the supremum is taken over all intervals of finite length. In
\cite{LS97}, the authors also introduce the \emph{discrete Hilbert
transform }as follows. For fixed $\epsilon>0$ and two sequences $\Gamma:=\left\{ \gamma_{n}\right\} _{n}$
and $\Sigma:=\left\{ \sigma_{n}\right\} _{n}$ satisfying $\left|\gamma_{n}-\sigma_{n}\right|=\epsilon$,
and $a=\left(a_{n}\right)_{n}$,
\[
\left(\mathcal{H}_{\Gamma,\Sigma}\left(a\right)\right)_{n}:=\sum_{j}\frac{a_{j}}{\gamma_{j}-\sigma_{n}}.
\]
According to \cite[Lemma 1]{LS97}), $\mathcal{H}_{\Gamma,\Sigma}$
is bounded from $l^{p}(w_{n})$ into itself if and only if $\left(w_{n}\right)_{n}$
satisfies the \emph{discrete Muckenhoupt condition 
\[
\left(\mathfrak{A}_{p}\right)\qquad\qquad\sup_{\underset{n>0}{k\in\mathbb{Z}}}\left(\frac{1}{n}\sum_{j=k+1}^{k+n}w_{j}\right)\left(\frac{1}{n}\sum_{j=k+1}^{k+n}w_{j}^{-1/\left(p-1\right)}\right)^{p-1}<\infty.
\]
}
\begin{defn}
A sequence $\Lambda\subset\mathbb{C}$ satisfies the condition $(LS)_{\tau,p}$
for $\tau>0$ and $1<p<\infty$, if the following
set of conditions hold:
\begin{itemize}
\item[(i)] $\:\forall a\in\mathbb{R},\;\Lambda\cap\mathbb{C}_{a}^{\pm}$ satisfies
the Carleson condition (\ref{eq:Carleson});
\item[(ii)] The sequence is \emph{relatively dense}:$\;\exists r>0$, $\forall x\in\mathbb{R}$,
\[
d(x,\Lambda):=\text{\ensuremath{\inf}}_{\lambda\in\Lambda}\left|x-\lambda\right|<r;
\]

\item[(iii)]  The limit 
\[
S(z)=\lim_{R\to\infty}\prod_{|\lambda|<R}\left(1-\frac{z}{\lambda}\right)
\]
exists and defines an entire function of exponential type $\tau$; 
\item[(iv)] The function $x\mapsto\left(\frac{|S(x)|}{d(x,\Lambda)}\right)^{p}$
satisfies $(A_{p})$.
\end{itemize}
\end{defn}

Note that if $0\in\Lambda$, then the corresponding factor in $(iii)$
reduces to $z$. In order to not complicate the notation we shall assume
in all what fallows that $0\not\in\Lambda$ which we can do without
loss of generality (for instance, by shifting the sequence). We are now
in a position to state the Lyubarski-Seip theorem \cite[Theorem 1]{LS97}.
\begin{thm}
\label{thm:(Lyubarskii-Seip,-1997)}(Lyubarskii-Seip). Let $\Lambda\subset\mathbb{C}$,
$\tau>0$ and $1<p<\infty$. The following assertions are equivalent.

$(1)$ $R_{\Lambda}$ is an isomorphism from $PW_{\tau}^{p}$ onto
$l^{p}\left(e^{-p\tau|\text{\emph{Im}}(\lambda)|}\left(1+|Im(\lambda)|\right)\right);$

$(2)$ $\Lambda$ satisfies $(LS)_{\tau,p}$.\end{thm}
\begin{rem}
The condition $(iv)$ can be replaced by the condition $(iv)'$
\begin{itemize}
\item[(iv)'] There is a relatively dense subsequence $\Gamma=\left(\gamma_{n}\right)_{n}\subset\Lambda$
such that the sequence $\left(\left|S'\left(\gamma_{n}\right)\right|^{p}\right)_{n}$
satisfies the discrete Muckenhoupt condition $\left(\mathfrak{A}_{p}\right)$. 
\end{itemize}
\end{rem}
\bigskip

The aim of this paper is to generalize the Lyubarskii-Seip result 
to 
finite unions of Carleson sequences. 
In the case of Hardy spaces, this problem has been solved by Vasyunin
\cite{Va84} and Hartmann \cite{Ha96b} and involves divided differences.

As mentioned previously, in the case $p=2$ the Fourier transform allows
to express our main result Theorem \ref{main result} in terms of bases 
of exponentials in $L^{2}$ thereby generalizing a result by
Avdonin and Ivanov \cite[Theorem 3]{AI01}.

\bigskip

This paper is organized as follows. The next section will be devoted
to divided differences.  Section 3 deals with $N-$Carleson
sequences. We will state our main result after some technical constructions
in the fourth section. For an easier reading, we have postponed the proofs 
of Section 4 to the fifth section. Finally, in the last
section we will discuss the necessity of the $N-$Carleson condition
with an appropriate definition of the trace $PW_{\tau}^{p}|\Lambda$.

\bigskip

A final word on notation. If $\delta$ is a metric on $\Omega$,
we will denote by $D_{\delta}\left(x,\eta\right)$ the ball (relatively
to $\delta$) with center $x\in\Omega$ and radius $\eta>0$, and
$\text{diam}_{\delta}(E)$ the $\delta-$diameter of $E$. We shortly
write $\text{diam}(E)$ and $D\left(x,\eta\right)$ when $\delta$
is the Euclidian distance. 
If $\omega=\left(\omega_{n}\right)_{n\geq1}$
is a sequence of strictly positive numbers and $1\leq p<\infty$,
we denote by $l^{p}\left(\omega\right)$ or $l^{p}\left(\omega_{n}\right)$
the space
\[
l^{p}\left(\omega\right):=\left\{ a=\left(a_{n}\right)_{n\geq1}:\:\sum_{n\geq1}\left|a_{n}\right|^{p}\omega_{n}<\infty\right\} .
\]

\section{Divided Differences}

Divided differences appear in many results about interpolation or
bases of exponentials (see e.g. \cite{Va84}, \cite{Ha96b}, \cite{BNO96}
or \cite{AI01}). Here we will give the definitions and some properties
that we will need later on. We recall that the (non-normalized)
Blaschke factors in a half-plane $\mathbb{C}_{a}^{\pm}$ are given
by
\[
b_{\mu}^{\pm,a}(z)=\frac{z-\mu}{z-\overline{\mu}-2ia}.
\]
(The formula is actually the same for the upper and the lower half-plane).
The associated \emph{pseudohyperbolic distance} will be denoted by
\[
\rho_{\pm,a}(z,\mu):=\left|b_{\mu}^{\pm,a}(z)\right|.
\]
For $\mathbb{C}^{+}$, we will write $b_{\mu}=b_{\mu}^{+,0}$ and
use $\rho$ for $\rho_{+,0}$ and $\rho_{-,0}$.

The definitions and properties below are stated and proved in $\mathbb{C}^{+}$
but are obviously valid for any half-plane $\mathbb{C}_{a}^{\pm}$.
\begin{defn}
Let $\Gamma:=\left\{ \mu_{i}:1\leq i\leq|\Gamma|<\infty\right\} \subset\mathbb{C}^{+}$.
For a finite set $a=\left\{ a_{i}\right\} _{1\leq i\leq|\Gamma|}$,
we define the sequence of \emph{(pseudohyperbolic) divided
differences }of $a$ relatively to $\Gamma$ as follows
\[
\Delta_{\Gamma}^{0}(a_{i}):=a_{i},\qquad\Delta_{\Gamma}^{1}(a_{i},a_{j}):=\frac{a_{j}-a_{i}}{b_{\mu_{i}}(\mu_{j})},
\]
and
\[
\Delta_{\Gamma}^{k}(a_{i_{1}},...,a_{i_{k+1}}):=\frac{\Delta_{\Gamma}^{k-1}(a_{i_{1}},...,a_{i_{k-1}},a_{i_{k+1}})-\Delta_{\Gamma}^{k-1}(a_{i_{1}},...,a_{i_{k}})}{b_{\mu_{i_{k}}}(\mu_{i_{k+1}})}.
\]
We will need to estimate the divided differences when $\Gamma$ lies in
a compact set $K\subset\mathbb{C}^{+}$ and $a=\left\{ f(\mu):\;\mu\in\Gamma\right\} $
for $f$ an analytic function bounded in $K$. Here $K$ is supposed
to be the closure of a non empty open connected set. 
By $f\in H^{\infty}(K)$ we mean that $f$ is holomorphic in the interior
of $K$ and 
\[
\left\Vert f\right\Vert _{\infty,K}:=\sup_{z\in K}\left|f(z)\right|<\infty.
\]
\end{defn}
\begin{lem}
\label{Local estimates DD}Suppose that $\Gamma$ lies in a compact
set $K$ with the properties mentioned above, and assume that there
exists $\eta>0$ such that $\rho(\Gamma,\partial K)\geq\eta$. Then,
for each function $f\in H^{\infty}(K)$, we have
\[
\left|\Delta_{\Gamma}^{j}\left(f(\mu^{(j+1)})\right)\right|\leq\left(\frac{2}{\eta}\right)^{j}\prod_{k=0}^{j}\left(\frac{1}{1-\frac{k}{2M}}\right)\left\Vert f\right\Vert _{\infty,K}
\]
where 
\[
\mu^{(j+1)}=\left(\mu_{1},...,\mu_{j+1}\right)\text{ and }f\left(\mu^{(j+1)}\right)=\left(f(\mu_{1}),...,f(\mu_{j+1})\right).
\]
\end{lem}
\begin{proof}
Set
\[
A_{j}:=\left\{ z\in K:\quad\rho(z,\partial K)\geq\frac{j}{2N}\eta\right\} ,\qquad0\leq j\leq N-1.
\]
We show by induction over $j$ that for every $z\in A_{j}$,
\[
\left|\Delta_{\Gamma}^{j}\left(f(\mu^{(j)},z)\right)\right|\leq c_{j}\left\Vert f\right\Vert _{\infty,K}
\]
with 
\[
c_{j}=\left(\frac{2}{\eta}\right)^{j}\prod_{k=0}^{j}\left(\frac{1}{1-\frac{k}{2M}}\right).
\]
Since $\Gamma\subset A_{N-1}\subset...\subset A_{1}\subset A_{0}$,
the result will follow. The claim is obviously true for $j=0$. Now, the function 
\[
z\mapsto\Delta_{\Gamma}^{j+1}\left(f\left(\mu^{(j+1)},z\right)\right)
\]
is holomorphic on $A_{j+1}$ and by the maximum principle and the
definition of divided differences, we have for $z\in A_{j+1},$
\begin{equation}
\left|\Delta_{\Gamma}^{j+1}\left(f\left(\mu^{(j+1)},z\right)\right)\right|\leq\underset{\xi\in\partial A_{j+1}}{\sup}\left|\frac{\Delta_{\Gamma}^{j}\left(f\left(\mu^{(j)},\xi\right)\right)-\Delta_{\Gamma}^{j}\left(f(\mu^{(j+1)})\right)}{\rho(\xi,\mu_{j+1})}\right|.\label{maj induction}
\end{equation}
Let $\xi\in\partial A_{j+1}$. It is possible to find a point $\zeta\in\partial K$
such that 
\[
\rho(\zeta,\xi)=\left(\frac{j+1}{2N}\right)\eta
\]
and so, since $\mu_{j+1}\in\Gamma$ and $\rho(\Gamma,\partial K)\geq\eta$,
we have, by the triangle inequality, 
\begin{equation}
\rho(\xi,\mu_{j+1})\geq\rho(\zeta,\mu_{j+1})-\rho(\xi,\zeta)\geq\eta\left(1-\frac{j+1}{2N}\right).\label{eq:min dist induction}
\end{equation}
From (\ref{maj induction}), (\ref{eq:min dist induction}) and the
induction hypothesis, we finally obtain
\[
\left|\Delta_{\Gamma}^{j+1}\left(f(\mu^{(j+1)},\xi)\right)\right|\leq\frac{2}{\eta}\left(\frac{1}{1-\frac{j+1}{2N}}\right)c_{j}\left\Vert f\right\Vert _{\infty,K}
\]
which gives the required estimate.
\end{proof}
The next lemma will be important in the sequel; we can define
a rational Newton type interpolating function which interpolates the
values $\left\{ a(\mu):\;\mu\in\Gamma\right\} $ on $\Gamma$.
\begin{lem}
\label{Interpolation Polynom}The holomorphic function
\[
P_{\Gamma,a}(z):=\sum_{k=1}^{|\Gamma|}\Delta_{\Gamma}^{k-1}\left(a(\mu^{(k)})\right)\prod_{l=1}^{k-1}b_{\mu_{l}}(z)
\]
satisfies
\[
P_{\Gamma,a}(\mu)=a(\mu),\qquad\mu\in\Gamma.
\]

\end{lem}
The proof is quite straightforward (see also \cite[p.80]{Ha96a}).
\begin{rem}
Divided differences with respect to pseudohyperbolic metric can
be found in \cite{BNO96,Ha96b,Va84}. 
We will also need
\emph{euclidian divided differences}:
\[
\square_{\Gamma}^{0}:=a_{i},\qquad\square_{\Gamma}^{1}\left(a_{i},a_{j}\right):=\frac{a_{j}-a_{i}}{\mu_{j}-\mu_{i}},
\]
and
\[
\square_{\Gamma}^{k}\left(a_{i_{1}},..,a_{i_{k+1}}\right):=\frac{\square_{\Gamma}^{k-1}\left(a_{i_{1}},..,a_{i_{k-1}},a_{i_{k+1}}\right)-\square_{\Gamma}^{k-1}\left(a_{i_{1}},..,a_{i_{k}}\right)}{\mu_{k+1}-\mu_{k}}.
\]

\end{rem}

\section{$N-$Carleson sequences}
\begin{defn}
Let $N\geq1$ be a natural number. A sequence $\Lambda\subset\mathbb{C}_{a}^{\pm}$
is called a $N-$\emph{Carleson sequence} if it is possible to find
a partition
\[
\Lambda=\bigcup_{i=1}^{N}\Lambda^{i}
\]
 such that, for every $i=1,...,N$, the sequence $\Lambda^{i}$ satisfies
the Carleson (\ref{eq:Carleson}) condition in $\mathbb{C}_{a}^{\pm}$.
\end{defn}

Note that the number $N$ is not uniquely defined.\\

Let us make a link between the $N-$Carleson condition and the \emph{Generalized
Carleson condition}, also called Carleson-Vasyunin condition (see
e.g. \cite{Ni86} and references therein). The following result has
originally been stated in $\mathbb{D}$ (see \cite[Proposition 3.1]{Ha96b})
but can easily be translated to any half-plane $\mathbb{C}_{a}^{\pm}$.
\begin{prop}
\label{Ncarleson-CG} Let $\Lambda$ be a sequence of complex numbers,
lying in $\mathbb{C}_{a}^{\pm}$. The following assertions are equivalent

$(i)$ $\Lambda$ is $N-$Carleson in $\mathbb{C}_{a}^{\pm}$;

$(ii)$ There exists $\delta>0$ and a sequence of Blaschke products
$\left(B_{n}\right)_{n\geq1}$ such that $\sup_{n}\deg B_{n}\leq N$,
$\Lambda=\bigcup_{n}\sigma_{n}$, with $\sigma_{n}:=\left\{ \lambda\in\mathbb{C}_{a}^{\pm}:\: B_{n}\left(\lambda\right)=0\right\} $
and $\left(B_{n}\right)_{n\geq1}$ satisfies the \emph{Generalized Carleson condition}
\begin{equation}
\left|B(z)\right|>\delta\inf_{n\geq1}\left|B_{n}(z)\right|,\qquad z\in\mathbb{\mathbb{C}}_{a}^{\pm},\label{(CG)}
\end{equation}
where $B$ denotes the Blaschke product associated to $\Lambda$.
\end{prop}
Observe that if $\Lambda$ satisfies $(ii)$,
then, for $\left(\lambda,\mu\right)\in\sigma_{n}\times\sigma_{m}$
($n\neq m$), we have $\rho\left(\sigma_{n},\sigma_{m}\right)\geq\delta$
and thus
\[
\inf_{n\neq m}\rho\left(\sigma_{n},\sigma_{m}\right)\geq\delta>0.
\]

\begin{rem}
\label{rem:diameters}The subsets $\sigma_{n}$ can for instance be
obtained as intersections $\tau_{n}^{\epsilon}\cap\Lambda$ 
where $\tau_{n}^{\epsilon}$ are the connected components
of $L(B,\epsilon):=\left\{ z:\;|B(z)|<\epsilon\right\} $ and $\epsilon$
is small enough. Moreover, choosing
$\epsilon$ in a suitable way, it is possible to assume that the pseudohyperbolic
diameter of $\sigma_{n}$ is arbitrarily small.\end{rem}
\begin{prop}
Let $\Lambda=\left\{ \lambda_{n}:\: n\geq1\right\} $ be an $N-$Carleson
sequence in $\mathbb{C}_{a}^{\pm}$. There exists $\eta>0$ such that
every connected component of $\bigcup_{n\geq1}D_{\rho}\left(\lambda_{n},\eta\right)$
admits at most $N$ elements.\end{prop}
\begin{rem}
\label{rem:inclusionrectangle}We can deduce from the previous proposition
that if $\Lambda$ is $N-$Carleson in $\mathbb{C}_{a}^{\pm}$ (or
equivalently satisfies condition $(ii)$ of Proposition \ref{Ncarleson-CG}),
it is possible to construct a sequence of rectangles of $\mathbb{C}_{a}^{\pm}$
defined by
\[
R_{n}=\text{Rect}\left(z_{n},L_{n},l_{n}\right)=\left\{ x+iy\in\mathbb{C}_{a}^{\pm}:\:\left|x-x_{n}\right|\leq\frac{L_{n}}{2},\:\left|y-y_{n}\right|\leq\frac{l_{n}}{2}\right\} 
\]
with $L_{n},l_{n}>0$ and $z_{n}=x_{n}+iy_{n}$. These rectangles
satisfy the following properties:

\begin{equation}
\sigma_{n}\subset R_{n},\qquad n\geq1;\label{rect1}
\end{equation}
\begin{equation}
L_{n}\asymp l_{n}\asymp\left|y_{n}-a\right|\asymp d\left(\partial R_{n},\mathbb{R}+ia\right),\qquad n\geq1;\label{rect2}
\end{equation}
\begin{equation}
0<\inf_{n\geq1}\rho\left(\sigma_{n},\partial R_{n}\right)\leq\sup_{\underset{\lambda\in\sigma_{n}}{n\geq1}}\rho\left(\lambda,\partial R_{n}\right)<\infty;\label{rect3}
\end{equation}

and finally, since the diameter of $\sigma_{n}$ can be chosen arbitrarily
small by Remark \ref{rem:diameters}, we can suppose the $R_{n}$
disjoints and even 
\begin{equation}
\inf_{n\not=k}\rho\left(R_{n},R_{k}\right)>0.\label{rect4}
\end{equation}

\end{rem}
Let $\Lambda$ be $N-$Carleson in $\mathbb{C}_{a}^{\pm}$ and $1<p<\infty$.
From Proposition \ref{Ncarleson-CG}, we can write 
\[
\Lambda=\bigcup_{n\geq1}\sigma_{n},
\]
with in particular $\left|\sigma_{n}\right|\leq N$. We will construct
divided differences relatively to $\sigma_{n}$. We set
\[
\sigma_{n}=\left\{ \lambda_{n,k}:\:1\leq k\leq\left|\sigma_{n}\right|\right\} \text{ and }\lambda_{n}^{\left(k\right)}=\left(\lambda_{n,1},...,\lambda_{n,k}\right).
\]
We choose, in an arbitrarily way, $\lambda_{n,0}$ in $\sigma_{n}$
and introduce, for $a=\left(a\left(\lambda\right)\right)_{\lambda\in\Lambda}\in\mathbb{C}^{\Lambda}$,
\[
\left\Vert a\right\Vert _{X_{\pm a}^{p}\left(\Lambda\right)}:=\left(\sum_{n\geq1}\left|\text{Im}(\lambda_{n,0})-a\right|\sum_{k=1}^{|\sigma_{n}|}\left|\Delta_{\sigma_{n}}^{k-1}\left(a\left(\lambda_{n}^{(k)}\right)\right)\right|^{p}\right)^{\frac{1}{p}}
\]
and the space
\[
X_{\pm a}^{p}\left(\Lambda\right):=\left\{ a\in\mathbb{C}^{\Lambda}:\;\left\Vert a\right\Vert _{X_{\pm a}^{p}\left(\Lambda\right)}<\infty\right\} .
\]

Observe that for every $\lambda\in\sigma_{n}$, $1\asymp\left|\text{Im}\left(\lambda\right)-a\right|/\left|\text{Im}\left(\lambda_{n,0}\right)-a\right|$
and so the definition of $X_{\pm a}^{p}\left(\Lambda\right)$ does
not depend on the choice of $\lambda_{n,0}$. The following result
was originally stated in $\mathbb{D}$ (see \cite{Ha96b}) but it
is not hard to check that it holds in $\mathbb{C}_{a}^{\pm}$. The
reader will find  details in \cite[p. 92]{Gau11}.
\begin{thm}
\label{thm Hartmann}(Hartmann). Let $\Lambda$ be $N-$Carleson in
$\mathbb{C}_{a}^{\pm}$ and $1<p<\infty$. Then, $R_{\Lambda}$ is
continuous and surjective from $H^{p}\left(\mathbb{C}_{a}^{\pm}\right)$
onto $X_{\pm a}^{p}\left(\Lambda\right)$.
\end{thm}

\section{\label{sec:Main-Result}Main Result}

Let $\Lambda$ be a sequence in the complex plane. In this section
we 
assume that there is an integer $N\geq1$ such that
for every $a\in\mathbb{R}$, the sequence 
\[
\Lambda_{a}^{\pm}:=\Lambda\cap\mathbb{C}_{a}^{\pm}
\]
is $N-$Carleson in the corresponding half-plane. 
Note that the partitions discussed in the previous section were
adapted to sequences in a half-plane.
Here, we will start discussing a {}``right'' partition of $\Lambda$ taking into 
account the fact that $\Lambda$ lies in the whole complex plane

\subsection{\label{An-adapted-partition}An adapted partition}

From our above discussions 
it is possible to write
\[
\Lambda_{a}^{\pm}=\bigcup_{n\geq1}\sigma_{n,a}^{\pm},
\]
where $\left(B_{\sigma_{n,a}^{\pm}}^{\pm,a}\right)_{n}$ satisfies
the generalized Carleson condition in the corresponding half-plane
$\mathbb{C}_{a}^{\pm}$ ($B_{\sigma_{n,a}^{\pm}}^{\pm,a}$ being the
Blaschke product in $\mathbb{\mathbb{C}}_{a}^{\pm}$ vanishing on
$\sigma_{n,a}^{\pm}$). To simplify the notation, we will omit $a$
if $a=0$ and write
\[
\sigma_{n}:=\left\{ \begin{array}{lc}
\sigma_{n+1}^{+}, & n\geq0\\
\sigma_{n}^{-}, & n<0
\end{array}\right..
\]
The reader might notice that $\sigma_{n}^{+}$ and $\sigma_{m}^{-}$
can come very close for certain values of $n$ and $m$. This issue
will be fixed below. Let us distinguish the sets of points close to
the real axis and the ones far away from it. Let us fix $\epsilon>0$
for all what follows. We can assume that
\[
\rho_{0}:=\sup_{n\in\mathbb{Z}}\text{diam}_{\rho}\left(\sigma_{n}\right)<\frac{\epsilon}{2}.
\]
(Observe that $\rho_{0}$ is well defined by the Generalized Carleson
condition). Next introduce
\[
M_{\epsilon,\infty}:=\left\{ n\in\mathbb{Z}:\:\sigma_{n}\cap\left\{ \left|\text{Im}(z)\right|<\epsilon\right\} =\emptyset\right\} ,
\]
\[
\Lambda_{\epsilon,\infty}:=\bigcup_{n\in M_{\epsilon,\infty}}\sigma_{n}
\]
(corresponding to the points for which the corresponding set $\sigma_{n}$
does not interset the previous strip) and
\[
\Lambda_{\epsilon}:=\Lambda\setminus\Lambda_{\epsilon,\infty}.
\]
Notice that $\Lambda_{\epsilon}$ contains the points of $\Lambda$
lying in the real axis
and moreover
\[
\Lambda_{\epsilon}\subset\left\{ z\in\mathbb{C}:\:\left|\text{Im}(z)\right|<3\epsilon\right\} .
\]
Indeed, if $\lambda\in\Lambda_{\epsilon}$ and $\lambda\not\in\mathbb{R}$,
then there is $n_{\lambda}\in\mathbb{Z}\setminus M_{\epsilon,\infty}$
such that $\lambda\in\sigma_{n_{\lambda}}$. Hence, it is possible
to find $\mu\in\sigma_{n_{\lambda}}$ such that $\left|\text{Im}\left(\mu\right)\right|<\epsilon$.
It follows that
\begin{eqnarray*}
\left|\lambda-\mu\right| & = & \frac{\left|\lambda-\mu\right|}{\left|\lambda-\overline{\mu}\right|}\left|\lambda-\overline{\mu}\right|\\
 & \leq & \rho_{0}\left(2\left|\text{Im}\left(\mu\right)\right|\right)+\left|\lambda-\mu\right|\\
 & \leq & \frac{3}{2}\epsilon^{2}<\frac{3}{2}\epsilon,
\end{eqnarray*}
which implies that $\left|\text{Im}\left(\lambda\right)\right|<5\epsilon/2$.
Now, since $\Lambda_{\epsilon}$ is contained in a strip, parallel
to the real axis, of finite width and is $N-$Carleson in $\mathbb{C}_{-3\epsilon}^{+}$
, $\Lambda_{\epsilon}$ breaks up into a disjoint union
\[
\Lambda_{\epsilon}=\bigcup_{n\geq1}\sigma_{n}^{'}
\]
with
\[
\rho_{0}^{'}:=\sup_{n\geq1}\text{diam}\left(\sigma_{n}^{'}\right)<\frac{\epsilon}{2}
\]
and moreover, for some $\delta>0$, the subsets
\[
\Omega_{n}:=\left\{ z\in\mathbb{C}:\prod_{\lambda\in\sigma_{n}^{'}}\left|z-\lambda\right|\leq\delta\right\} ,\qquad n\geq1,
\]
satisfy
\begin{equation}
\inf_{n\neq m}d\left(\Omega_{n},\Omega_{m}\right)>0.\label{eq:distance Omega positive}
\end{equation}
This is possible in view of Remarks \ref{rem:diameters} and \ref{rem:inclusionrectangle}.
It follows that we can write $\Lambda$ as the following disjoint
union
\[
\Lambda=\left(\bigcup_{n\in M_{\epsilon,\infty}}\sigma_{n}\right)\cup\left(\bigcup_{n\geq1}\sigma_{n}'\right)=:\bigcup_{n\in\mathbb{Z}}\tau_{n}.
\]
Now that the partition is done, it is possible to construct divided
differences. Since we will need both definitions of divided differences,
we set
\[
\tilde{\Delta}_{\tau_{n}}:=\left\{ \begin{array}{cl}
\Delta_{\tau_{n}} & \text{if }\exists k\text{ s.t. }\tau_{n}=\sigma_{k}\\
\square_{\tau_{n}} & \text{if }\exists k\text{ s.t. }\tau_{n}=\sigma_{k}^{'}
\end{array}\right..
\]
It is now possible to introduce a space of sequences that will be,
assuming some hypotheses on $\Lambda$, the range of $R_{\Lambda}$.
Naturally, we write
\[
\tau_{n}=\left\{ \lambda_{n,k}:\:1\leq k\leq\left|\sigma_{n}\right|\right\} \text{ and }\lambda_{n}^{(k)}:=\left(\lambda_{n,1},...,\lambda_{n,k}\right).
\]
As previously, we choose, in an arbitrarily way, $\lambda_{n,0}\in\tau_{n},$
for every $n\in\mathbb{Z}$. We define, for $1<p<\infty$, 
\[
X_{\tau,\epsilon}^{p}(\Lambda):=\left\{ a=\left(a(\lambda)\right)_{\lambda\in\Lambda}:\:\left\Vert a\right\Vert _{X_{\tau,\epsilon}^{p}(\Lambda)}<\infty\right\} ,
\]
with
\[
\left\Vert a\right\Vert _{X_{\tau,\epsilon}^{p}(\Lambda)}^{p}:=\sum_{n\in\mathbb{Z}}\left(1+\left|\text{Im}(\lambda_{n,0})\right|\right)\sum_{k=1}^{|\tau_{n}|}\left|\tilde{\Delta}_{\tau_{n}}^{k-1}\left(ae^{\pm i\tau\cdot}\left(\lambda_{n}^{(k)}\right)\right)\right|^{p},
\]
and
\[
e^{\pm i\tau\lambda}=\left\{ \begin{array}{cc}
e^{i\tau\lambda} & \text{ if }\lambda\in\tau_{n},\quad n\in N_{+},\\
e^{-i\tau\lambda} & \text{if }\lambda\in\tau_{n},\quad n\in N_{-},
\end{array}\right.
\]
where
\[
N_{+}:=\left\{ n\in\mathbb{Z}:\:\tau_{n}\cap\left(\mathbb{C}^{+}\cup\mathbb{R}\right)\not=\emptyset\right\} 
\]
and 
\[
N_{-}:=\mathbb{Z}\setminus N_{+}.
\]
(The factor $e^{\pm i\tau\lambda}$ does not really matter close to
$\mathbb{R}$.) Next proposition will be proved in Section \ref{Proof-of-Theorem main}.
\begin{prop}
\label{nec rel dens}If there exists $\epsilon>0$ such that $R_{\Lambda}$
is an isomorphism between $PW_{\tau}^{p}$ and $X_{\tau,\epsilon}^{p}\left(\Lambda\right)$
then $\Lambda$ is relatively dense, i.e. there exists $r>0$ such
that for every $x\in\mathbb{R}$, $d\left(x,\Lambda\right)<r$.
\end{prop}
It follows from the conclusion of the previous proposition that the
relative density is necessary. Thus, we will assume in all what follows
that $\Lambda$ is relatively dense:
\[
\exists r>0,\;\forall x\in\mathbb{R},\; d(x,\Lambda)<r.
\]

Still relative to the previous partition of $\Lambda$, we introduce,
for $n\geq1$, the products
\[
p_{n}\left(x\right):=\prod_{\lambda\in\tau_{n}}\left|x-\lambda\right|
\]
which permit us to define the function
\[
d_{N}\left(x\right):=\inf_{n\in\mathbb{Z}}p_{n}\left(x\right),\qquad x\in\mathbb{R}.
\]

\begin{rem}
\label{rem:dN}From the definition of the function $d_{N}$, we can
do the following observations.
\begin{itemize}
\item $(1)$ The relative density condition implies that 
\[
\sup_{x\in\mathbb{R}}\: d_{N}(x)\leq\left(r+\delta_{0}^{'}\right)^{N}<\infty,
\]
where
\[
\delta_{0}^{'}:=\inf_{n\not=m}d(\sigma_{n}^{'},\sigma_{m}^{'})>0.
\]

\item $(2)$ It is clear that, in the definition of $d_{N}$, the infimum
is actually a minimum. So, for each $x\in\mathbb{R}$, there is $n_{x}\in\mathbb{Z}$
such that $d_{N}(x)=p_{n_{x}}(x)$. It is not difficult to see that
\[
\inf_{x\in\mathbb{R}}\inf_{m\not=n_{x}}p_{m}(x)\ge\left(\frac{\delta_{0}^{'}}{2}\right)^{N}>0.
\]

\item $(3)$ Using the relative density, a similar reasoning as the one
that can be used to show $(2)$ yields that, with an other partition
(and in particular with an other choice of $\epsilon$), the function
obtained is equivalent to $d_{N}$.
\end{itemize}
\end{rem}

\subsection{The theorem}
\begin{defn}
Let $\Lambda$ be $N-$Carleson in every half-plane and relatively
dense. We say that $\Lambda$ satisfies the conditions $\left(H_{N}\right)_{\tau,p}$
(for $\tau>0$ and $1<p<\infty$) if 
\begin{itemize}
\item $(i)$ The limit 
\[
S(z):=\lim_{R\to\infty}\prod_{|\lambda|<R}\left(1-\frac{z}{\lambda}\right)
\]
exists and defines an entire function of exponential type $\tau$.
\item $(ii)$ The function $x\mapsto\left(\frac{|S(x)|}{d_{N}(x)}\right)^{p}$
satisfies the (continuous) Muckenhoupt condition $(A_{p})$.
\end{itemize}
\end{defn}
The reader would notice that, in view of Remark \ref{rem:dN}$-(3)$,
the definition of the conditions $\left(H_{N}\right)_{\tau,p}$ do
no depend on the partition of $\Lambda$. 
\begin{thm}
\label{main result}Let $N\geq1$, $\tau>0$, $1<p<\infty$ and $\Lambda$
be $N-$Carleson in every half-plane and relatively dense (for some
$r>0$). Then, the restriction operator $R_{\Lambda}$ is an isomorphism
from $PW_{\tau}^{p}$ onto $X_{\tau,r}^{p}(\Lambda)$ if and only
if $\Lambda$ satisfies $\left(H_{N}\right)_{\tau,p}$.\end{thm}
\begin{rem}
We will see in the following that $\left(H_{N}\right)_{\tau,p}-(ii)$
can be replaced by $(ii)'$, which is
\begin{itemize}
\item $(ii)'$ There exists a subsequence $\Gamma=\left\{ \gamma_{n}:\: n\geq1\right\} \subset\Lambda$,
still relatively dense, such that, if $\sigma_{\gamma_{n}}$ is the
set containing $\gamma_{n}$, the sequence 
\[
\left(\frac{\left|S'\left(\gamma_{n}\right)\right|^{p}}{\underset{\underset{\lambda\not=\gamma_{n}}{\lambda\in\sigma_{\gamma_{n}}}}{\prod}\left|\gamma_{n}-\lambda\right|^{p}}\right)_{n\geq1}
\]
satisfies the discrete Muckenhoupt condition $(\mathfrak{A}_{p})$.
\end{itemize}
\end{rem}
It is clear that for $N=1$, $d_{1}(x)=d(x,\Lambda)$ and $(H_{1})_{\tau,p}$
with the Carleson condition and the relative density corresponds exactly
to the $(LS)_{\tau,p}$ conditions. The proof of Theorem \ref{main result}
will be done in Section \ref{Proof-of-Theorem main}.
\begin{rem}
The choice of $\epsilon=r$ in our construction ensures that, for
every $x\in\mathbb{R}$, $\tau_{n_{x}}=\sigma_{n_{x}}^{'}$ and permits
us to avoid tedious considerations but the conclusion or Theorem \ref{main result}
is still true with any choice of $\epsilon>0$. 
\end{rem}
\bigskip

We will discuss below the necessity of the $N-$Carleson condition
in Theorem \ref{Thm : necessity}. In Theorem \ref{main result},
the definition of the range of $R_{\Lambda}$ definitely depends on
the partition of $\Lambda$ which is possible because of the $N-$Carleson
condition. In Section \ref{sec:About-the-Carleson}, we will construct
a space without the \emph{a priori} assumption that $\Lambda$ is
$N-$Carleson in every half-plane.

\clearpage

\section{\label{Proof-of-Theorem main}Proofs}

\subsection{Proof of Proposition \ref{nec rel dens}}
\begin{proof}
Let us suppose to the contrary that there exists a real sequence $\left\{ x_{j}\right\} _{j\geq1}$
and a sequence of positive numbers $\left\{ r_{j}\right\} _{j\geq1}$
such that $r_{j}\to\infty$, $j\to\infty$ and 
\[
B(x_{j},r_{j})\cap\Lambda=\emptyset.
\]
We consider the functions of $PW_{\tau}^{p}$ defined by 
\[
f_{j}(z):=\frac{\sin\tau\left(z-x_{j}\right)}{\tau\left(z-x_{j}\right)},\quad z\in\mathbb{C},\; j\geq1.
\]
Since $R_{\Lambda}$ is an isomorphism, we obtain that 
\[
1\asymp\left\Vert f_{j}\right\Vert _{p}^{p}\asymp\left\Vert R_{\Lambda}f_{j}\right\Vert _{X_{\tau,\epsilon}^{p}\left(\Lambda\right)}^{p}.
\]
We will show that $\left\Vert R_{\Lambda}f_{j}\right\Vert _{X_{\tau}^{p}\left(\Lambda\right)}^{p}\longrightarrow0$,
$j\to\infty$, which implies the required contradiction. From the
definition, we have 
\[
\left\Vert R_{\Lambda}f_{j}\right\Vert _{X_{\tau,\epsilon}^{p}\left(\Lambda\right)}^{p}=\sum_{n\geq1}\left(1+\left|\text{Im}\left(\lambda_{n,0}\right)\right|\right)\sum_{k=1}^{\left|\tau_{n}\right|}\left|\tilde{\Delta}_{\tau_{n}}^{k-1}\left(f_{j}e^{\pm i\tau\cdot}\left(\lambda_{n}^{(k)}\right)\right)\right|^{p}.
\]
Using Lemme \ref{Local estimates DD} (see \cite[p. 95]{Gau11} for
details), we can see that, for every $n\geq1$ and every $1\leq k\leq\left|\tau_{n}\right|$,
\[
\left|\tilde{\Delta}_{\tau_{n}}^{k-1}\left(f_{j}e^{\pm i\tau\cdot}\left(\lambda_{n}^{(k)}\right)\right)\right|^{p}\lesssim\frac{1}{\left|\lambda_{n,0}-x_{j}\right|^{p}},
\]
which implies
\[
\left\Vert R_{\Lambda}f_{j}\right\Vert _{X_{\tau,\epsilon}^{p}\left(\Lambda\right)}^{p}\lesssim\sum_{n\geq1}\frac{1+\left|\text{Im}\left(\lambda_{n,0}\right)\right|}{\left|\lambda_{n,0}-x_{j}\right|^{p}}.
\]
On the other hand, $p>1$ and so we can find $\alpha>0$ such that
$p-\alpha>1$. Recall that $\left|\lambda_{n,0}-x_{j}\right|\geq r_{j}$
and let us write
\[
\left\Vert R_{\Lambda}f_{j}\right\Vert _{X_{\tau,\epsilon}^{p}\left(\Lambda\right)}^{p}\lesssim\frac{1}{r_{j}^{\alpha}}\sum_{n\geq1}\frac{1+\left|\text{Im}\left(\lambda_{n,0}\right)\right|}{\left|\lambda_{n,0}-x_{j}\right|^{p-\alpha}}.
\]
We split this sum in two parts, writing $\left\{ \lambda_{n,0}:\: n\geq1\right\} =A^{+}\cup A^{-}$,
where 
\[
A^{+}\subset\left(\mathbb{C}^{+}\cup\mathbb{R}\right)\subset\mathbb{C}_{-\frac{1}{2}}^{+}
\]
and 
\[
A^{-}\subset\mathbb{C}^{-}\subset\mathbb{C}_{\frac{1}{2}}^{-}.
\]
Since $r_{j}\to\infty$, $j\to\infty$, we obtain $\left|\lambda_{n,0}-x_{j}\right|\asymp\left|\lambda_{n,0}-x_{j}\pm i\right|$.
It follows that the functions
\[
g^{\pm}:z\mapsto\frac{1}{z-x_{j}\pm i}\in H^{p-\alpha}\left(\mathbb{C}_{\mp\frac{1}{2}}^{\pm}\right).
\]
Now, $A^{\pm}$ is Carleson in $\mathbb{C}_{\mp\frac{1}{2}}^{\pm}$
, thus 
\[
\sum_{\lambda\in A^{\pm}}\frac{1+\left|\text{Im}\left(\lambda\right)\right|}{\left|\lambda-x_{j}\pm i\right|^{p-\alpha}}=\sum_{\lambda\in A^{\pm}}\frac{1+\left|\text{Im}\left(\lambda\right)\right|}{\left|g^{\pm}\left(\lambda\right)\right|^{p-\alpha}}\lesssim\left\Vert g\right\Vert _{H^{p-\alpha}\left(\mathbb{C}_{\mp\frac{1}{2}}^{\pm}\right)}^{p-\alpha}\lesssim1.
\]
We finally obtain that 
\[
\left\Vert R_{\Lambda}f_{j}\right\Vert _{X_{\tau,\epsilon}^{p}\left(\Lambda\right)}^{p}\lesssim\frac{1}{r_{j}^{\alpha}}\to0,j\to\infty,
\]
which is the required contradiction and ends the proof.
\end{proof}

\subsection{Proof of Theorem \ref{main result}}

The proof of Theorem \ref{main result} follows the main ideas of
Lyubarskii and Seip's paper but needs an important technical work
to characterize this more general case.

\subsubsection{Paley-Wiener Spaces}

We will need some well known facts about Paley-Wiener spaces that
we recall here. First, we have the Plancherel-Poly\`a inequality
(see e.g. \cite{Le96} or \cite[p. 95]{Se04}).
\begin{prop}
(Plancherel-Poly\`a).\label{plancherel-polya}Let $f\in PW_{\tau}^{p}$
and $a\in\mathbb{R}$. Then, 
\[
\int_{-\infty}^{+\infty}|f(x+ia)|^{p}dx\leq e^{\tau p|a|}\left\Vert f\right\Vert _{p}^{p}.
\]

\end{prop}
It follows that for every $f\in PW_{\tau}^{p}$, the function $z\mapsto e^{i\tau z}f(z)$
belongs to $H_{+}^{p}$. It also follows that translation is an isomorphism
from $PW_{\tau}^{p}$ onto itself. The second fact is a pointwise
estimate; there exists a constant $C=C(p)$ such that for every $f\in PW_{\tau}^{p}$,
we have
\begin{equation}
|f(z)|\leq C\left\Vert f\right\Vert _{p}\left(1+|\text{Im}(z)\right)^{-\frac{1}{p}}e^{\tau|\text{Im}(z)|},\qquad z\in\mathbb{C}.\label{Pointwise PW}
\end{equation}

\subsubsection{Necessary conditions}

Let us do the construction of subsection \ref{An-adapted-partition}
with $\epsilon=r$ and suppose that $R_{\Lambda}$ is an isomorphism
between $PW_{\tau}^{p}$ and $X_{\tau,\epsilon}^{p}\left(\Lambda\right)$.
The necessity of $\left(H_{N}\right)-(i)$ can be shown exactly as
in \cite{LS97} and so we do not prove it here. We first show that
the condition $(ii)'$ is necessary. Then, with a technical lemma,
adapted from \cite{LS97}, we prove that $(ii)'$ implies $(ii)$. 

Since $R_{\Lambda}$ is bijective, for each $\lambda\in\Lambda$,
there is a unique function $f_{\lambda}\in PW_{\tau}^{p}$ such that
\[
f_{\lambda}(\mu)=\left\{ \begin{array}{cc}
1, & \text{if }\mu=\lambda\\
0, & \text{if }\mu\not=\lambda
\end{array}\right..
\]
As in \cite{LS97}, it can be shown that $f_{\lambda}$ only vanishes
on $\Lambda\setminus\left\{ \lambda\right\} $ and that $f_{\lambda}$
is of exponential type $\tau$ (if its type was $\tau'<\tau$ then
considering the function $e^{i\left(\tau-\tau'\right)\left(\cdot-\lambda\right)}f_{\lambda}$,
we would obtain a contradiction with the injectivity of $R_{\Lambda}$).
Moreover, $z\mapsto\left(z-\lambda\right)f_{\lambda}(z)$ is a function
of the Cartwright Class $\mathcal{C}$ vanishing exactly on $\Lambda$
(see e.g. \cite{Le96} for definition and general results on $\mathcal{C}$).
Hence, since $S$ is also of exponential type $\tau$, $S(z)=c_{\lambda}\left(z-\lambda\right)f_{\lambda}(z)$,
$z\in\mathbb{C},$ or 
\[
f_{\lambda}(z)=\frac{S(z)}{S'(\lambda)\left(z-\lambda\right)}.
\]

For each $n\geq1$, the holomorphic function
\[
g_{n}:\: z\mapsto\frac{S(z)}{\underset{\lambda\in\sigma_{n}^{'}}{\prod}\left(z-\lambda\right)}
\]
does not vanish in $\Omega_{n}$ (see Formula \ref{eq:distance Omega positive}).
Moreover, choosing $\lambda_{n,0}^{'}\in\sigma_{n}^{'}$, 
\[
g_{n}(\lambda_{n,0})=\frac{S'\left(\lambda_{n,0}^{'}\right)}{\underset{\underset{\lambda\not=\lambda_{n,0}}{\lambda\in\sigma_{n}^{'}}}{\prod}\left(\lambda_{n,0}^{'}-\lambda\right)}.
\]
Hence, it follows from the maximum and the minimum principle that
\[
\inf_{\xi\in\partial\Omega_{n}}\left|\frac{S(\xi)}{\underset{\lambda\in\sigma_{n}^{'}}{\prod}\left(\xi-\lambda\right)}\right|\leq\left|\frac{S'\left(\lambda_{n,0}^{'}\right)}{\underset{\underset{\lambda\not=\lambda_{n,0}}{\lambda\in\sigma_{n}^{'}}}{\prod}\left(\lambda_{n,0}^{'}-\lambda\right)}\right|\leq\sup_{\xi\in\partial\Omega_{n}}\left|\frac{S(\xi)}{\underset{\lambda\in\sigma_{n}^{'}}{\prod}\left(\xi-\lambda\right)}\right|.
\]
From the intermediate values theorem, we deduce the existence of a
point $\theta_{n}\in\partial\Omega_{n}$ such that
\begin{equation}
\left|S(\theta_{n})\right|=\delta\frac{\left|S'(\lambda_{n,0}^{'})\right|}{\underset{\underset{\lambda\not=\lambda_{n,0}}{\lambda\in\sigma_{n}^{'}}}{\prod}\left|\lambda_{n,0}^{'}-\lambda\right|}=:\delta\omega_{n}.\label{eq:s theta n}
\end{equation}
Let us consider now a subsequence $\Gamma:=\left(\gamma_{n}\right)_{n\geq1}$
of $\left\{ \lambda_{n,0}^{'}:\: n\geq1\right\} $ which is still
relatively dense and such that 
\[
\inf_{n\geq1}\left(\text{Re}\left(\gamma_{n+1}\right)-\text{Re}\left(\gamma_{n}\right)\right)>0.
\]
We define $\sigma_{\gamma_{n}}$ as the set containing $\gamma_{n}$.
The sequence $\Theta:=\left(\theta_{n}\right)_{\geq1}$ denotes the
previous $\theta_{n}$, corresponding to $\gamma_{n}$, and for $n\geq1$,
we set
\[
\omega_{n}:=\left|\frac{S'\left(\gamma_{n}\right)}{\underset{\underset{\lambda\neq\gamma_{n}}{\lambda\in\sigma_{\gamma_{n}}}}{\prod}\left(\gamma_{n}-\lambda\right)}\right|
\]
so that
\[
\left|S\left(\theta_{n}\right)\right|=\delta\omega_{n}.
\]
We show that the discrete Hilbert transform $\mathcal{H}_{\Gamma,\Theta}$
is bounded from $l^{p}(\omega)$ into itself. Indeed, let $\left(a_{n}\right)_{n\geq1}$
be a finite sequence of $l^{p}(\omega)$. Then, the sequence
\[
a(\lambda):=\left\{ \begin{array}{ll}
a_{n}S'\left(\gamma_{n}\right) & ,\:\text{if }\lambda=\gamma_{n}\\
0 & ,\:\text{if }\lambda\in\Lambda\setminus\Gamma
\end{array}\right.
\]
belongs to $X_{\tau}^{p}(\Lambda)$ because, if $\gamma_{k}=\lambda_{n,0}^{'}=\lambda_{n,\left|\sigma_{\gamma_{n}}\right|}$
is choosen as the {}``last'' point of $\sigma_{n}^{'}$, 
\[
\tilde{\Delta}_{\sigma_{n}}^{k-1}\left(ae^{i\tau\cdot}\left(\lambda_{n}^{(k)}\right)\right)=0,\qquad k<|\sigma_{n}|
\]
and 
\[
\left|\tilde{\Delta}_{\sigma_{n}}^{|\sigma_{n}|-1}\left(ae^{i\tau\cdot}\left(\lambda_{n}^{(|\sigma_{n}|)}\right)\right)\right|=\frac{\left|a_{n}S'(\lambda_{n,0}^{'})\right|e^{-\tau\left|\text{Im}\left(\lambda_{n,0}^{'}\right)\right|}}{\underset{\underset{\lambda\not=\lambda_{n,0}^{'}}{\lambda\in\sigma_{n}^{'}}}{\prod}\left|\lambda-\lambda_{n,0}^{'}\right|}.
\]
Thus, from (\ref{eq:s theta n}), we obtain, observing that $1+\left|\text{Im}(\lambda_{n,0}^{'})\right|$
and $\left|e^{i\tau\lambda}\right|$, $\lambda\in\sigma_{n}^{'}$,
are comparable to a constant since $\sigma_{n}^{'}$ is close to $\mathbb{R}$,
\begin{eqnarray}
\left\Vert a\right\Vert _{X_{\tau}^{p}(\Lambda)}^{p} & = & \sum_{n}\left(1+\left|\text{Im}(\lambda_{n,0}^{'})\right|\right)\left|\tilde{\Delta}_{\sigma_{n}^{'}}^{|\sigma_{n}^{'}|-1}\left(ae^{i\tau\cdot}\left(\lambda_{n}^{(|\sigma_{n}^{'}|)}\right)\right)\right|^{p}\nonumber \\
 & \asymp & \sum_{n}\left(\frac{\left|a_{n}S'(\lambda_{n,0}^{'})\right|}{\underset{\underset{\lambda\not=\lambda_{n,0}^{'}}{\lambda\in\sigma_{n}^{'}}}{\prod}\left|\lambda_{n,0}^{'}-\lambda\right|}\right)^{p}=\sum_{n}\omega_{n}^{p}\left|a_{n}\right|^{p}.\label{eq:normea}
\end{eqnarray}
 So, let $f\in PW_{\tau}^{p}$ be the (unique) solution of the interpolation
problem $f|\Lambda=a$. Notice that, since $R_{\Lambda}$ is an isomorphism
onto $X_{\tau}^{p}(\Lambda)$, then 
\begin{equation}
\left\Vert f\right\Vert _{p}^{p}\lesssim\left\Vert a\right\Vert _{X_{\tau}^{p}(\Lambda)}^{p}.\label{eq:normefnormea}
\end{equation}
This function is of the form $f(z)=\sum_{j}a_{j}\frac{S(z)}{z-\gamma_{j}}$
and so, with (\ref{eq:s theta n}) we have
\[
\sum_{n}\left|f(\theta_{n})\right|^{p}=\sum_{n}\left|\sum_{j}a_{j}\frac{S\left(\theta_{n}\right)}{\theta_{n}-\gamma_{j}}\right|=\sum_{n}\left|S\left(\theta_{n}\right)\right|\left|\sum_{j}\frac{a_{j}}{\theta_{n}-\gamma_{j}}\right|
\]
and, from the construction of $\Theta$, we obtain
\begin{equation}
\sum_{n}\left|f(\theta_{n})\right|^{p}=\delta^{p}\sum_{n}\omega_{n}^{p}\left|\left(\mathcal{H}_{\Gamma,\Theta}(\left(a_{j}\right)_{j\geq1})\right)_{n}\right|^{p}.\label{eq:hilbert1}
\end{equation}
On the other hand, the Poly\`a inequality (see \cite[Lecture 20]{Le96}),
and the inequalities (\ref{eq:normefnormea}) and (\ref{eq:normea})
give 
\begin{equation}
\sum_{n}\left|f(\theta_{n})\right|^{p}\lesssim\left\Vert f\right\Vert _{p}^{p}\lesssim\left\Vert a\right\Vert _{X_{\tau}^{p}(\Lambda)}^{p}\lesssim\sum_{n}\omega_{n}^{p}\left|a_{n}\right|^{p}.\label{eq:hilbert2}
\end{equation}
From (\ref{eq:hilbert1}) and (\ref{eq:hilbert2}), we deduce that
$\mathcal{H}_{\Gamma,\Theta}$ is bounded from $l^{p}(\omega^{p})$
into itself. Using a slight modified version of \cite[Lemma 1]{LS97},
we can conclude that the weight $\left(\omega_{n}^{p}\right)_{n\geq1}$
satisfies the discrete Muckenhoupt condition $\left(\mathfrak{A}_{p}\right)$.
\begin{rem}
\label{rem: distance entre les gamman} It follows from the weak density
condition ($\left(H_{N}\right)-(i)$), the Genralized Carleson condition
(\ref{(CG)}) on $\left(B_{\sigma_{\gamma_{n}}}\right)_{n}$ and the
growth of the sequence $\left(\text{Re}\left(\gamma_{n}\right)\right)_{n}$
that we have $\text{Re}(\gamma_{n+1})-\text{Re}\left(\gamma_{n}\right)\leq3\epsilon$.
This implies that
\[
\delta_{0}^{'}\leq\left|\gamma_{n}-\gamma_{n+1}\right|\leq4\epsilon.
\]

\end{rem}
Now, in order to prove $(iii)$, we use the following lemma, adapted
from \cite[Lemma 2]{LS97}. 
\begin{lem}
Suppose $x\in\mathbb{R}$ and $\text{Re}(\gamma_{n})\leq x\leq\text{Re}(\gamma_{n+1})$.
Then, there exists an $\alpha=\alpha(x)\in[0,1]$ such that
\[
\omega_{n}^{\alpha}\omega_{n+1}^{1-\alpha}\asymp\frac{\left|S(x)\right|}{d_{N}(x)},
\]
uniformly with respect to $x\in\mathbb{R}$.
\end{lem}
Assuming this lemma to hold, $(iii)$ follows from $(iii)'$ and the
inequality $t^{\alpha}s^{1-\alpha}\leq t+s$, $t,s>0$ and $\alpha\in[0,1]$
(we still refer to \cite{Gau11} for details).
\begin{proof}
For $x\in[\text{Re}(\gamma_{n}),\text{Re}(\gamma_{n+1})]$, we set
$N(x):=\left\{ n:\: d(\sigma_{n}^{'},x)<\epsilon\right\} $ and 
\[
\Lambda(x):=\left(\bigcup_{n\in N(x)}\sigma_{n}^{'}\right)\cup\sigma_{\gamma_{n}}\cup\sigma_{\gamma_{n+1}}.
\]
Notice that $\sigma_{\gamma_{n}}$ and $\sigma_{\gamma_{n+1}}$ may
be subsets of $\bigcup_{n\in N(x)}\sigma_{n}^{'}$. Observe also that
since $\Lambda$ is a finite union of Carleson sequences, we have
\[
\sup_{x\in\mathbb{R}}\left|N(x)\right|<\infty.
\]

For $\alpha\in[0,1]$, we want to show that $\vartheta\asymp1,$ where
\[
\vartheta:=\frac{\omega_{n}^{\alpha}\omega_{n+1}^{1-\alpha}d_{N}(x)}{|S(x)|},
\]
and $x\not\in\Lambda$ (this is not restrictive since the expression
extends continuously to $\Lambda$). From the definition of $S$,
we have that
\[
S'(\lambda)=-\frac{1}{\lambda}\underset{\underset{\mu\not=\lambda}{\mu\in\Lambda}}{\prod}\left(1-\frac{\lambda}{\mu}\right),\qquad\lambda\in\Lambda.
\]
In order to not overcharge notation, all infinite products occurring
below will be understood as symmetric limits of finite products:
\[
\prod_{\lambda\in\Lambda}a\left(\lambda\right)=\lim_{R\to\infty}\prod_{\left|\lambda\right|\leq R}a\left(\lambda\right).
\]

Thus, 
\[
\negthickspace\negthickspace\vartheta=\left(\frac{\left|\frac{1}{\gamma_{n}}\underset{\lambda\in\Lambda\setminus\{\gamma_{n}\}}{\prod}\left(1-\frac{\gamma_{n}}{\lambda}\right)\right|^{\alpha}\left|\frac{1}{\gamma_{n+1}}\underset{\lambda\in\Lambda\setminus\{\gamma_{n+1}\}}{\prod}\left(1-\frac{\gamma_{n+1}}{\lambda}\right)\right|^{1-\alpha}d_{N}(x)}{\underset{\lambda\in\Lambda}{\prod}\left(1-\frac{x}{\lambda}\right)\underset{\lambda\in\sigma_{\gamma_{n}}\setminus\left\{ \gamma_{n}\right\} }{\prod}\left|\lambda-\gamma_{n}\right|^{\alpha}\underset{\lambda\in\sigma_{\gamma_{n+1}}\setminus\left\{ \gamma_{n+1}\right\} }{\prod}\left|\lambda-\gamma_{n+1}\right|^{1-\alpha}}\right).
\]
For $\lambda\in\Lambda\setminus\left\{ \gamma_{n},\gamma_{n+1}\right\} $,
\[
\frac{\left|1-\frac{\gamma_{n}}{\lambda}\right|^{\alpha}\left|1-\frac{\gamma_{n+1}}{\lambda}\right|^{1-\alpha}}{\left|1-\frac{x}{\lambda}\right|}=\frac{\left|\lambda-\gamma_{n}\right|^{\alpha}\left|\lambda-\gamma_{n+1}\right|^{1-\alpha}}{\left|x-\lambda\right|}.
\]
Note also that for the remaining two points $\gamma_{n},\gamma_{n+1}$
we have:
\[
\frac{\left|\frac{1}{\gamma_{n}}\left(1-\frac{\gamma_{n}}{\gamma_{n+1}}\right)\right|^{\alpha}\left|\frac{1}{\gamma_{n+1}}\left(1-\frac{\gamma_{n+1}}{\gamma_{n}}\right)\right|^{1-\alpha}}{\left|\left(1-\frac{x}{\gamma_{n}}\right)\left(1-\frac{x}{\gamma_{n+1}}\right)\right|}=\frac{\left|\gamma_{n+1}-\gamma_{n}\right|^{\alpha}\left|\gamma_{n}-\gamma_{n+1}\right|^{1-\alpha}}{\left|\gamma_{n}-x\right|\left|\gamma_{n+1}-x\right|}.
\]
Now, we split $\vartheta$ in two products $\vartheta=\Pi_{1}(x)\cdot\Pi_{2}(x)$
corresponding essentially to zeros in $\Lambda(x)$ and zeros in $\Lambda\setminus\Lambda(x)$
($d_{N}(x)$ appearing in $\Pi_{1}$):
\begin{eqnarray*}
\Pi_{1}(x): & = & \frac{\underset{\lambda\in\Lambda(x)\setminus\left\{ \gamma_{n}\right\} }{\prod}\left|\lambda-\gamma_{n}\right|^{\alpha}\underset{\lambda\in\Lambda(x)\setminus\left\{ \gamma_{n+1}\right\} }{\prod}\left|\lambda-\gamma_{n+1}\right|^{1-\alpha}d_{N}(x)}{\underset{\lambda\in\Lambda(x)}{\prod}\left|\lambda-x\right|\underset{\lambda\in\sigma_{\gamma_{n}}\setminus\left\{ \gamma_{n}\right\} }{\prod}\left|\lambda-\gamma_{n}\right|^{\alpha}\underset{\lambda\in\sigma_{\gamma_{n+1}}\setminus\left\{ \gamma_{n+1}\right\} }{\prod}\left|\lambda-\gamma_{n+1}\right|^{1-\alpha}}\\
 & = & \frac{\underset{\lambda\in\Lambda(x)\setminus\sigma_{\gamma_{n}}}{\prod}\left|\lambda-\gamma_{n}\right|^{\alpha}\underset{\lambda\in\Lambda(x)\setminus\sigma_{\gamma_{n+1}}}{\prod}\left|\lambda-\gamma_{n+1}\right|^{1-\alpha}d_{N}(x)}{\underset{\lambda\in\Lambda(x)}{\prod}\left|\lambda-x\right|}
\end{eqnarray*}
and
\[
\Pi_{2}(x):=\prod_{\lambda\in\Lambda\setminus\Lambda(x)}\left(\frac{\left|\lambda-\gamma_{n}\right|^{\alpha}\left|\lambda-\gamma_{n+1}\right|^{1-\alpha}}{\left|\lambda-x\right|}\right).
\]
We can write
\[
\Pi_{1}(x)=\left(\frac{\underset{\lambda\in\sigma_{\gamma_{n+1}}}{\prod}\left|\lambda-\gamma_{n}\right|^{\alpha}\underset{\lambda\in\sigma_{\gamma_{_{n}}}}{\prod}\left|\lambda-\gamma_{n+1}\right|^{1-\alpha}d_{N}(x)}{\underset{\sigma_{\gamma_{n}}\cup\sigma_{\gamma_{n+1}}}{\prod}\left|x-\lambda\right|}\right)
\]
\[
\qquad\qquad\quad\times\left(\prod_{\Lambda(x)\setminus\left(\sigma_{\gamma_{n}}\cup\sigma_{\gamma_{n+1}}\right)}\frac{\left|\lambda-\gamma_{n}\right|^{\alpha}\left|\lambda-\gamma_{n+1}\right|^{1-\alpha}}{\left|x-\lambda\right|}\right)
\]
and notice that if $\lambda\in\Lambda(x)\setminus\left(\sigma_{\gamma_{n}}\cup\sigma_{\gamma_{n+1}}\right)$,
then $\lambda\in\sigma_{l}^{'}$ for a suitable $l\in N(x)$, so that
\[
1\lesssim d(\sigma_{\gamma_{n}},\sigma_{l}^{'})\leq\left|\lambda-\gamma_{n}\right|\leq2\rho_{0}^{'}+2\epsilon\lesssim1
\]
and, in view of Remark \ref{rem: distance entre les gamman}, for
$\lambda\in\sigma_{\gamma_{n}}$ and $\mu\in\sigma_{\gamma_{n+1}}$,
we have 
\[
\left|\lambda-\gamma_{n+1}\right|\asymp1\text{ and }\left|\mu-\gamma_{n}\right|\asymp1.
\]
These three relations imply that
\[
\Pi_{1}(x)\asymp\frac{d_{N}(x)}{\underset{\lambda\in\Lambda(x)}{\prod}\left|x-\lambda\right|}.
\]
Now, let $n_{x}$ be such that $d_{N}(x)=p_{n_{x}}(x)$ (we refer
to Remark \ref{rem:dN}). Clearly $n_{x}\in N(x)$. Note also that
for $\lambda\in\sigma_{m}^{'}$, $m\in N(x)$, we have $\left|\lambda-x\right|\le d\left(\sigma_{m}^{'},x\right)+\text{diam}\left(\sigma_{m}^{'}\right)\le\epsilon+\rho_{0}^{'}$.
Hence 
\[
\frac{1}{\left(\epsilon+\rho_{0}^{'}\right)^{\left|N(x)\right|-1}}\leq\frac{d_{N}(x)}{\underset{\lambda\in\Lambda(x)}{\prod}\left|x-\lambda\right|}=\frac{1}{\underset{\lambda\in\Lambda(x)\setminus\sigma_{n_{x}}}{\prod}\left|\lambda-x\right|}\leq\left(\frac{2}{\delta_{0}^{'}}\right)^{N\cdot\left(\left|N(x)\right|-1\right)}
\]
and, from the end of Remark \ref{rem:dN}, we obtain that 
\[
\Pi_{1}(x)\asymp1.
\]
The relation
\[
\Pi_{2}(x)\asymp1
\]
is shown exactly in the same way as in \cite{LS97}, using the  $N-$Carleson
condition. The lemma is proved.
\end{proof}

\subsubsection{Sufficient conditions. }

We show the converse of the theorem in two parts; first, the injectivity
of $R_{\Lambda}$ and then its surjectivity.

\bigskip

Let $f\in PW_{\tau}^{p}$ such that $f(\lambda)=0,$ $\lambda\in\Lambda$.
We want to show that $f\equiv0$. Let us introduce $\phi:=f/S$. It
can be shown that $\phi$ is an entire function of exponential type
$0$ (see \cite[pp. 96-98]{Gau11} for details). The idea of the proof,
given by Lyubarskii and Seip in \cite{LS97}, is to bound $\phi$
by a constant on the imaginary axis and to use a Phragmen-Lindel\"of
theorem to obtain that $\phi$ is a constant. Then, for integrability
reasons, the only possible value for the constant will be zero.

\bigskip

We will proceed as follows: since $\phi$ is analytic, it is bounded
on the compact $\left[-2i\epsilon,2i\epsilon\right]$. In order to
bound $\phi$ on $i\mathbb{R}\setminus\left[-2i\epsilon,2i\epsilon\right]$,
we will use a lower estimate for $S$ in a certain area of $\mathbb{C}$.
Let us introduce 
\[
A_{n}:=\left\{ z\in\mathbb{C}:\:\left|\text{Im}(z)\right|\ge2\epsilon,\;\rho\left(\lambda_{n,0},z\right)<2\rho_{0}<\epsilon\right\} ,\quad n\in\mathbb{Z}.
\]
We begin to show that for $z\in\left(\mathbb{C}_{2\epsilon}^{+}\cup\mathbb{C}_{-2\epsilon}^{-}\right)\setminus\left(\bigcup_{n}A_{n}\right)$,
\begin{equation}
\left|S(z)\right|\gtrsim e^{\tau\left|\text{Im}(z)\right|}\left(\left|\text{Im}(z)\right|\right)^{\frac{1}{q}}\left(1+|z|\right)^{-1}.\label{eq:S lower estimate}
\end{equation}
Indeed, let us introduce
\[
S_{1}(z):=\left(S/B_{\epsilon}\right)(z),
\]
where 
\[
B_{\epsilon}(z):=\prod_{\lambda\in\Lambda_{\epsilon}}\left(c_{\lambda}\frac{z-\lambda}{z-\overline{\lambda}+3i\epsilon}\right),
\]
is the Blaschke product in $\mathbb{C}_{-\frac{3}{2}\epsilon}^{+}$
associated to $\Lambda_{\epsilon}$, and $c_{\lambda}$ is the unimodular
normalizing constant which ensures the convergence of the Blaschke
product (we do not need the explicit value here). Let $x\in\mathbb{R}$.
Observe that for $n\geq1$ and $\lambda\in\sigma_{n_{x}}^{'}$ , we
have 
\[
\left|x-\overline{\lambda}\right|=\left|x-\lambda\right|\le\epsilon+\text{diam}\left(\sigma_{n_{x}}^{'}\right)\le\epsilon+\rho_{0}^{'}\lesssim1.
\]
Hence, 
\[
\left|x-\overline{\lambda}+3i\epsilon\right|\asymp1.
\]
It follows from these inequalities that 
\[
\left(\prod_{\lambda\in\sigma_{n_{x}}^{'}}\left|\frac{x-\lambda}{x-\overline{\lambda}+3i\epsilon}\right|\right)\asymp d_{N}(x).
\]
Writing 
\[
\left|B_{\epsilon}(x)\right|=\left(\prod_{\lambda\in\sigma_{n_{x}}^{'}}\left|\frac{x-\lambda}{x-\overline{\lambda}+3i\epsilon}\right|\right)\left(\prod_{\lambda\in\Lambda_{\epsilon}\setminus\sigma_{n_{x}}^{'}}\left|\frac{x-\lambda}{x-\overline{\lambda}+3i\epsilon}\right|\right)
\]
and using the fact that $\Lambda_{\epsilon}$ is $N-$Carleson in
$\mathbb{C}_{-\frac{3}{2}\epsilon}^{+}$, we have then that 
\begin{equation}
\left|B_{\epsilon}(x)\right|\asymp d_{N}(x),\label{eq:Bepsilon equiv dN}
\end{equation}
and so $x\mapsto|S_{1}(x)|^{p}$ satisfies $\left(A_{p}\right)$. 

In particular, the function $z\mapsto e^{i\tau z}\frac{S_{1}(z)}{z+i}=e^{i\tau z}\frac{S\left(z\right)}{B_{\epsilon}\left(z+i\right)}$
belongs to $H_{+}^{p}$ and the function $z\mapsto e^{i\tau z}S_{1}(z)$
is a function of $\mathcal{N}^{+}$, the Smirnov Class in the upper
half-plane (for definition and general results, see e.g. \cite[A.4]{Ni02a}).
Hence, we can write
\[
S_{1}(z)=e^{-i\tau z}B_{1}(z)G_{1}(z),\qquad z\in\mathbb{C}^{+},
\]
where $B_{1}$ is the Blaschke product associated to $\Lambda^{+}\setminus\Lambda_{\epsilon}$
and $G_{1}$ is an outer function in $\mathbb{C}^{+}$(observe that
$e^{i\tau\cdot}S_{1}$ cannot contain any inner singular factor).
Thus, $x\mapsto|G_{1}(x)|^{p}$ satisfies $\left(A_{p}\right)$ or
equivalently, $x\mapsto|G_{1}(x)|^{-q}$ satisfies $\left(A_{q}\right)$,
with $\frac{1}{p}+\frac{1}{q}=1$. So, it follows from properties
of functions satisfying Muckenhoupt's $\left(A_{p}\right)$ condition,
that
\[
\phi_{G_{1}}:\: z\mapsto\frac{1}{G_{1}(z)(z+i)}\in H_{+}^{q}
\]
and, from well known estimates in $H_{+}^{q}$, we get
\[
\left|\phi_{G_{1}}\left(z\right)\right|\lesssim\frac{1}{\left(\text{Im}(z)\right)^{\frac{1}{q}}},
\]
and so, for $z\in\mathbb{C}^{+}$,
\[
\left|\frac{1}{G_{1}(z)}\right|\lesssim\left(1+|z|\right)\left(\text{Im}(z)\right)^{-\frac{1}{q}}.
\]
Moreover, because of the $N-$Carleson condition of $\Lambda^{+}\setminus\Lambda_{\epsilon}$,
we have that
\[
\left|B_{1}(z)\right|\gtrsim1,\qquad z\in\mathbb{C}^{+}\setminus\left(\bigcup_{n\geq0}A_{n}\right)
\]
 and so we do have the lower bound for $S_{1}$ stated in (\ref{eq:S lower estimate}).
We notice that $\left|S\left(z\right)\right|\asymp\left|S_{1}\left(z\right)\right|$,
$\text{Im}(z)>2\epsilon$ and so we have the same bound for $S$ in
$\mathbb{C}_{2\epsilon}^{+}$. A similar reasonning gives us the estimate
in $\mathbb{C}_{-2\epsilon}^{-}$.

Using now (\ref{Pointwise PW}) and (\ref{eq:S lower estimate}),
we have for $z\in\left(\mathbb{C}_{2\epsilon}^{+}\cup\mathbb{C}_{-2\epsilon}^{-}\right)\setminus\left(\bigcup_{n}A_{n}\right)$, 

\begin{eqnarray*} \left|\phi(z)\right| &=&\left|\frac{f(z)}{S(z)}\right| \lesssim \frac{\left(1+|z|\right)} { e^{\tau|\text{Im}(z)|} \left|\text{Im}(z)\right|^{\frac{1}{q} }} \frac{ e^{\tau|\text{Im}(z)| } } { \left(1+\left|\text{Im}(z)\right|\right)^{\frac{1}{p}} } \\ &\asymp& \frac{\left(1+|z|\right)} { \left|\text{Im}(z)\right|^{\frac{1}{q} } \left(1+\left|\text{Im}(z)\right|\right)^{\frac{1}{p}} } =:\psi(z). \end{eqnarray*}We
notice then that if $A_{n}\cap i\mathbb{R}\not=\emptyset$, then 
\[
A_{n}\subset S_{\pm}:=\left\{ z\in\mathbb{C^{\pm}}:\;\left|\frac{\text{Im}(z)}{\text{Re}(z)}\right|<\eta\right\} ,
\]
where $\eta$ is a suitable constant. Note that $S_{\pm}$ are Stolz
angles in $\mathbb{C}^{\pm}$ at $x=0$. Since $A_{n}$ is far from
$\mathbb{R}$ and has uniformly bounded pseudohyperbolic diameter,
every $A_{n}$ hitting the imaginary axis will be in the Stolz angle
$S_{+}$ or $S_{-}$. Obviously, there is some $M>0$ such that for
every $z\in\mathbb{C}_{\pm2\epsilon}^{\pm}\cap S_{\pm}$, we have
\[
\left|\psi(z)\right|\leq M.
\]
In particular, $\left|\phi(z)\right|\leq M$ for $z\in\partial A_{n}$
and by the maximum principle,
\[
\left|\phi(iy)\right|\leq M\text{ for }iy\in A_{n}\cap i\mathbb{R}.
\]
Hence, $\phi$ is uniformly bounded on $i\mathbb{R}$ and it follows,
by a Phragmen-Lindel\"of principle that $\phi\equiv K$, and $f=KS$.
Let us now show that $K=0$. Because $x\mapsto\left|S_{1}\left(x\right)\right|^{p}$
satisfies $\left(A_{p}\right)$ we have
\[
\int\left|S_{1}\left(x\right)\right|^{p}=\infty
\]
and, applying the Plancherel-Poly\`a inequality, we also have
\[
\int\left|S_{1}\left(x+2i\epsilon\right)\right|^{p}=\infty
\]
but $\left|S\left(x+2i\epsilon\right)\right|\asymp\left|S_{1}\left(x+2i\epsilon\right)\right|$,
so 
\[
\int\left|S_{1}\left(x+2i\epsilon\right)\right|^{p}=\infty.
\]
We apply again the Plancherel-Poly\`a inequality to obtain
\[
\int\left|S\left(x\right)\right|^{p}=\infty.
\]
From the fact that $f\in PW_{\tau}^{p}$, we have by definition that
$f\in L^{p}$ and since $f=\phi S=KS$, the only possibility is $K=0$
and so $f\equiv0$, which ends the proof of the injectivity of $R_{\Lambda}$.
Now, we can show the last part of the proof.

\bigskip

Let us consider a finite sequence $a=\left(a\left(\lambda\right)\right)_{\lambda\in\Lambda}$
and the solution of the interpolation problem $f\left(\lambda\right)=a\left(\lambda\right)$,
$\lambda\in\Lambda$, given by 
\[
f(z)=\sum_{\lambda\in\Lambda}a(\lambda)\frac{S(z)}{S'(\lambda)(z-\lambda)}.
\]
Since the sum is finite, $f$ is an entire function of type at most
$\tau$. We want to split this sum according to the localization of
the points of $\Lambda$. More precisely, we recall that we have the
decomposition $\Lambda=\bigcup_{n\in\mathbb{Z}}\tau_{n}$ and we have
already introduced
\[
N_{+}=\left\{ n\in\mathbb{Z}:\;\tau_{n}\cap\left(\mathbb{C}^{+}\cup\mathbb{R}\right)\not=\emptyset\right\} \text{ and }N_{-}=\mathbb{Z}\setminus N_{+}.
\]
We set
\[
\Lambda_{+}:=\bigcup_{n\in N_{+}}\tau_{n}\text{ and }\Lambda_{-}:=\bigcup_{n\in N_{-}}\tau_{n}=\Lambda\setminus\Lambda_{+}.
\]
(Observe that since $\text{diam}\left(\tau_{n}\right)<\frac{\epsilon}{2}$,
we have $\Lambda_{+}\subset\mathbb{C}_{-\frac{\epsilon}{2}}^{+}$).
Now, we can write $f=f^{+}+f^{-}$, with 
\[
f^{\pm}(z):=\sum_{\lambda\in\Lambda_{\pm}}a(\lambda)\frac{S(z)}{S'(\lambda)(z-\lambda)}=\sum_{n\in N_{\pm}}\sum_{\lambda\in\tau_{n}}a(\lambda)\frac{S(z)}{S'(\lambda)(z-\lambda)}.
\]
We want to estimate, separately, 
\begin{equation}
\inf\left\{ \left\Vert f^{\pm}-g\right\Vert _{p}:\; g\in PW_{\tau}^{p},\; g|\Lambda=0\right\} .\label{inf a estimer}
\end{equation}
Here we will only consider $f^{+}$, the method is the same for $f^{-}$.
In the following, $\beta$ will be the Blaschke product associated
to $\Lambda_{-\epsilon}^{+}:=\Lambda\cap\mathbb{C}_{-\epsilon}^{+}$
\[
\beta(z)=\prod_{\lambda\in\Lambda_{-\epsilon}^{+}}\left(c_{\lambda}\frac{z-\lambda}{z-\overline{\lambda}+2i\epsilon}\right),\qquad z\in\mathbb{C}_{-\epsilon}^{+},
\]
where again $c_{\lambda}$ is a suitable normalizing factor. For $z\in\mathbb{C}_{-\epsilon}^{+}$,
we write $S(z)=e^{-i\tau z}\beta(z)G(z)$. Observe that $\beta(0)=\prod_{\lambda\in\Lambda}c_{\lambda}\frac{\lambda}{\overline{\lambda}-2i\epsilon}$
(recall that we have assumed $0\not\in\Lambda$). Thus, we can write
\begin{eqnarray*}
G(z) & = & e^{i\tau z}S(z)\beta(z)^{-1}\\
 & = & e^{i\tau z}\prod_{\lambda\in\Lambda}\left(\frac{\lambda-z}{\lambda}\right)\prod_{\lambda\in\Lambda_{-\epsilon}^{+}}\left(c_{\lambda}\frac{z-\overline{\lambda}+2i\epsilon}{z-\lambda}\right)\\
 & = & \beta(0)^{-1}e^{i\tau z}\prod_{\tilde{\lambda}\in\tilde{\Lambda}}\left(1-\frac{z}{\tilde{\lambda}}\right),
\end{eqnarray*}
with $\tilde{\Lambda}:=\left(\Lambda\setminus\Lambda_{-\epsilon}^{+}\right)\cup\left(\overline{\Lambda_{-\epsilon}^{+}}-2i\epsilon\right)\subset\mathbb{C}_{-\epsilon}^{-}$.
The function $G$ is outer in $\mathbb{C}_{-\epsilon}^{+}$. As in
(\ref{eq:Bepsilon equiv dN}), we obtain $\left|\beta(x)\right|\asymp d_{N}(x)$.
In particular, we have $|G(x)|^{p}\in(A_{p})$. Let then be $\eta$
such that $\frac{\epsilon}{2}<\eta<\epsilon$. Since $\tilde{\Lambda}$
is the union (not necessarily disjoint) of two $N-$Carleson sequences
in $\mathbb{C}_{-\epsilon}^{-}$, and in particular
\[
\text{Im}\left(\tilde{\lambda}+i\eta\right)\leq\eta-\epsilon<0,\quad\tilde{\lambda}\in\tilde{\Lambda},
\]
(which implies in particular that every real $x$ is far from $\tilde{\Lambda}$),
we obtain that
\[
\left|G(x-i\eta)\right|=e^{\tau\eta}\left|G(x)\right|\left(\prod_{\tilde{\lambda}\in\tilde{\Lambda}}\left|\frac{x-\tilde{\lambda}-i\eta}{x-\tilde{\lambda}}\right|\right)\asymp|G(x)|.
\]

So $x\mapsto|G(x-i\eta)|^{p}$ also satisfies the Muckenhoupt condition
$\left(A_{p}\right)$. According to the Plancherel-Poly\`a inquality,
it is possible to estimate (\ref{inf a estimer}) on the axis $\left\{ \text{Im}(z)=-\eta\right\} $.

By duality arguments (see \cite[p. 576]{SS61} or \cite[p. 94]{Gau11}),
we need to estimate
\[
\sup_{\underset{\left\Vert h\right\Vert _{q}=1}{h\in H^{q}\left(\mathbb{C}_{-\eta}^{+}\right)}}N(h),
\]
with
\begin{eqnarray*}
N(h): & = & \left|\sum_{\lambda\in\Lambda^{+}}\frac{a(\lambda)}{S'(\Lambda)}\int\frac{G(x-i\eta)h(x-i\eta)}{x-i\eta-\lambda}dx\right|\\
 & = & \left|\sum_{\lambda\in\Lambda^{+}}\frac{a(\lambda)}{S'(\lambda)}\mathcal{H}(\tilde{G}\tilde{h})(\lambda+i\eta)\right|
\end{eqnarray*}
where $z\mapsto\tilde{G}(z)=G(z-i\eta)$ is an outer function in $\mathbb{C}^{+}$and
the function $z\mapsto\tilde{h}(z)=h(z-i\eta)$ belongs to $H_{+}^{q}$.
In order to compute $S'(\lambda)$, let us recall that
\[
S(z)=e^{-i\tau z}\beta(z)G(z),\qquad z\in\mathbb{C}_{-\eta}^{+}.
\]
For $\lambda\in\tau_{n}$, $n\in N_{+}$, we have
\[
S'(\lambda)=c_{\lambda}\frac{e^{-i\tau\lambda}}{\lambda-\overline{\lambda}+2i\epsilon}G(\lambda)\frac{\beta}{b_{\lambda}^{\epsilon}}(\lambda),
\]
where $b_{\lambda}^{\epsilon}(z)=c_{\lambda}\frac{z-\lambda}{z-\overline{\lambda}+2i\epsilon}$.
Using that $G(\lambda)=\tilde{G}(\lambda+i\eta)$, and setting 
\[
\psi:=\frac{\mathcal{H}(\tilde{G}\tilde{h})}{\tilde{G}}\text{ and }\alpha(\lambda):=a(\lambda)e^{i\tau\lambda},\quad\lambda\in\Lambda^{+},
\]
 where we recall that $\mathcal{H}$ denotes the Hilbert transform
(see \ref{eq:Hilbert transform def} for definition) the expression
becomes
\[
N(h)=\left|\sum_{n\in N_{+}}\sum_{\lambda\in\tau_{n}}\frac{\alpha(\lambda)\psi(\lambda+i\eta)}{\underset{\mu\not=\lambda}{\prod}b_{\mu}^{\epsilon}(\lambda)}\left(\lambda-\overline{\lambda}+2i\epsilon\right)\right|.
\]
Writing
\[
N_{+}=N_{\epsilon}\cup N_{\infty},\text{ with }N_{\epsilon}:=\left\{ n\in N_{+}:\:\tau_{n}\cap\left\{ \left|\text{Im}(z)\right|<\epsilon\right\} \not=\emptyset\right\} ,
\]
we set, with the help of the functions of Lemma \ref{Interpolation Polynom},
\[
P_{\tau_{n},\alpha}(z):=\sum_{k=1}^{|\tau_{n}|}\Delta_{\tau_{n}}^{k-1}\left(\alpha\left(\lambda_{n}^{(k)}\right)\right)\prod_{l=1}^{k-1}b_{\lambda_{n,l}}(z),\qquad n\in N_{\infty},
\]
\[
P_{\tau_{n},\alpha}(z):=\sum_{k=1}^{|\tau_{n}|}\square_{\tau_{n}}^{k-1}\left(\alpha\left(\lambda_{n}^{(k)}\right)\right)\prod_{l=1}^{k-1}\left(z-\lambda_{n,l}\right),\qquad n\in N_{\epsilon}
\]
and setting $\tilde{\tau}_{n}:=\tau_{n}+i\eta$
\[
Q_{\tilde{\tau}_{n},\psi}(z):=\sum_{k=1}^{|\tilde{\tau}_{n}|}\Delta_{\tilde{\tau}_{n}}^{k-1}\left(\psi(\lambda_{n,1}+i\eta,...,\lambda_{n,k}+i\eta)\right)\prod_{l=1}^{k-1}b_{\lambda_{n,l}+i\eta}(z).
\]
We notice that 
\[
N(h)=\left|\sum_{n\in N_{+}}\sum_{\lambda\in\tau_{n}}\frac{P_{\tau_{n},\alpha}(\lambda)Q_{\tilde{\tau}_{n},\psi}(\lambda+i\eta)}{\underset{\mu\not=\lambda}{\prod}b_{\mu}^{\epsilon}(\lambda)}\left(\lambda-\overline{\lambda}+2i\epsilon\right)\right|.
\]
Recall now that $\tau_{n}\subset R_{n}$, where $\left(R_{n}\right)_{n}$
are the disjoint rectangles (constructed here in the half-plane $\mathbb{C}_{-\eta}^{+}$
so that in particular satisfying $d\left(\partial R_{n},\mathbb{R}-i\eta\right)\asymp l_{n}\asymp L_{n}$)
introduced in Remark \ref{rem:inclusionrectangle}. (Note also that
here we have that $\Lambda_{+}\subset\mathbb{C}_{-\frac{\epsilon}{2}}^{+}$
and in particular, $\Lambda_{+}$ is far from $\mathbb{R}-i\eta$).
Then, if $\Gamma_{n}:=\partial R_{n}$, the function 
\[
z\mapsto h_{n}(z):=\frac{P_{\tau_{n},\alpha}(z)Q_{\tilde{\tau}_{n},\psi}(z+i\eta)}{\beta(z)}
\]
 is a meromorphic function in $\overset{\circ}{R}_{n}$ with simple
poles at $\lambda\in\tau_{n}$. Thus, the residue theorem implies
that 
\[
\int_{\Gamma_{n}}h_{n}(z)dz=2i\pi\sum_{\lambda\in\tau_{n}}\text{Res}(h_{n},\lambda)
\]
and 
\[
\text{Res}(h_{n},\lambda)=P_{\tau_{n},\alpha}(\lambda)Q_{\tilde{\tau}_{n},\psi}(\lambda+i\eta)\left(\frac{\beta}{b_{\lambda}^{\epsilon}}(\lambda)\right)^{-1}\cdot\left(\lambda-\overline{\lambda}+2i\epsilon\right).
\]
It follows that 
\[
N(h)=\left|\frac{1}{2i\pi}\sum_{n\in N_{+}}\int_{\Gamma_{n}}\frac{P_{\tau_{n},\alpha}\left(z\right)Q_{\tilde{\tau}_{n},\psi}\left(z+i\eta\right)}{\beta}dz\right|.
\]
Obviously $\left|b_{\lambda_{n,l}}(z)\right|\le1$. Observe also that
by condition (\ref{rect4}) of Remark \ref{rem:inclusionrectangle}
for $z\in\Gamma_{n}$, $n\in N_{\epsilon}$, we have that $\left|z-\lambda_{n,l}\right|$
is bounded by a fixed constant. Hence for every $n\in N_{+}$, 
\[
\left|P_{\tau_{n},\alpha}\right|\lesssim\sum_{k=1}^{|\tau_{n}|}\left|\tilde{\Delta}_{\tau_{n}}^{k-1}\left(\alpha\left(\lambda_{n}^{(k)}\right)\right)\right|.
\]
Also
\[
\left|Q_{\tilde{\tau}_{n},\psi}\right|\lesssim\sum_{k=1}^{|\tilde{\tau}_{n}|}\left|\Delta_{\tilde{\tau}_{n}}^{k-1}\left(\psi(\lambda_{n,1}+i\eta,...,\lambda_{n,k}+i\eta)\right)\right|,
\]
and we obtain that 
\[
N(h)\lesssim\sum_{n\in N_{+}}\left[\left(\int_{\Gamma_{n}}\left|\frac{dz}{\beta(z)}\right|\right)\left(\sum_{k=1}^{|\tau_{n}|}\left|\tilde{\Delta}_{\tau_{n}}^{k-1}\left(\alpha\right)\right|\right)\left(\sum_{l=1}^{|\tilde{\tau}_{n}|}\left|\Delta_{\tilde{\tau}_{n}}^{l-1}\left(\psi\right)\right|\right)\right].
\]
For $z\in\Gamma_{n}$, we see that
\begin{eqnarray*}
\left|\beta(z)\right| & = & \left(\prod_{\lambda\in\Lambda_{+}\setminus\tau_{n}}\left|\frac{z-\lambda}{z-\overline{\lambda}+2i\epsilon}\right|\right)\cdot\left(\prod_{\lambda\in\tau_{n}}\left|\frac{z-\lambda}{z-\overline{\lambda}+2i\epsilon}\right|\right)\\
 & =: & \Pi_{1}(z)\cdot\Pi_{2}(z).
\end{eqnarray*}
Since $\Lambda_{+}$ is $N-$Carleson in $\mathbb{C}_{-\epsilon}^{+}$,
it follows from the fact that $R_{n}$ is {}``far'' from $\tau_{k}$,
$k\not=n$ that 
\[
\Pi_{1}(z)\asymp1
\]
and from the fact that $R_{n}$ is {}``far'' from $\tau_{n}$ that
\[
\Pi_{2}(z)\asymp1.
\]
Hence, choosing arbitrarily $\lambda_{n,0}\in\tau_{n}$, the construction
of $R_{n}$ gives 
\[
\int_{\Gamma_{n}}\left|\frac{dz}{\beta(z)}\right|\lesssim\int_{\Gamma_{n}}\left|dz\right|\lesssim\text{Im}(\lambda_{n,0})+\eta\lesssim1+\left|\text{Im}\left(\lambda_{n,0}\right)\right|.
\]
Applying H\"older's inequality, we obtain
\begin{eqnarray*}
N(h) & \lesssim & \left(\sum_{n\in N_{+}}\left(1+\text{Im}(\lambda_{n,0})\right)\sum_{k=1}^{|\tau_{n}|}\left|\tilde{\Delta}_{\tau_{n}}^{k-1}\left(e^{i\tau\cdot}a\right)\right|^{p}\right)^{\frac{1}{p}}\\
 &  & \times\left(\sum_{n\in N_{+}}\text{Im}(\lambda_{n,0}+i\eta)\sum_{k=1}^{|\tilde{\tau}_{n}|}\left|\Delta_{\tilde{\tau}_{n}}^{k-1}\left(\psi\right)\right|^{q}\right)^{\frac{1}{q}}.
\end{eqnarray*}
Now, notice that by the Muckenhoupt condition on $|\tilde{G}|^{-q}$
and thus the boundedness of $\mathcal{H}$ on
\[
H_{+}^{q}\left(\left|\frac{1}{\tilde{G}}\right|^{q}\right):=\left\{ f\in\mathcal{N}^{+}:\: f_{|\mathbb{R}}\in L^{q}\left(\left|\frac{1}{\tilde{G}}\right|^{q}\right)\right\} ,
\]
($\mathcal{N}^{+}$ denotes the Smirnov class) we get that $\psi\in H_{+}^{q}$
and $\left\Vert \psi\right\Vert _{q}\lesssim\left\Vert \tilde{h}\right\Vert _{H_{+}^{q}}=1$.
But, since 
\[
\bigcup_{n\in N_{+}}\tilde{\tau}_{n}=\Lambda^{+}+i\eta
\]
 is in fact $N-$Carleson in $\mathbb{C}_{\eta-\frac{\epsilon}{2}}^{+}\subset\mathbb{C}^{+}$
and $\psi\in H_{+}^{q}$, Theorem \ref{thm Hartmann} implies that
\[
\left(\sum_{n\in N_{+}}\text{Im}(\lambda_{n,0}+i\eta)\sum_{k=1}^{|\tilde{\tau}_{n}|}\left|\Delta_{\tilde{\tau}_{n}}^{k-1}\left(\psi\right)\right|^{q}\right)^{\frac{1}{q}}\lesssim\left\Vert \psi\right\Vert _{H_{+}^{q}}\lesssim\left\Vert \tilde{h}\right\Vert _{H_{+}^{q}}=1.
\]
Finally, we obtain
\[
N(h)\lesssim\left(\sum_{n\in N_{+}}\left(1+\text{Im}(\lambda_{n,0})\right)\sum_{k=1}^{|\tau_{n}|}\left|\tilde{\Delta}_{\tau_{n}}^{k-1}\left(e^{i\tau\cdot}a\right)\right|^{p}\right)^{\frac{1}{p}}=\left\Vert a\right\Vert _{X_{\tau,\epsilon}^{p}(\Lambda)},
\]
which ends the proof.

\section{\label{sec:About-the-Carleson}About the $N-$Carleson condition}

It is clear that the definition of $X_{\tau,\epsilon}^{p}\left(\Lambda\right)$
depends on the $N-$Carleson hypothesis, and more precisely for the
construction of the groups $\tau_{n}$. In this last section, we show
that in a certain way, the $N-$Carleson condition is necessary. 

It will be convenient to introduce the distance function 
\[
\delta(z,\xi):=\frac{\left|z-\xi\right|}{1+\left|z-\overline{\xi}\right|},\qquad z,\xi\in\mathbb{C},
\]
which expresses that locally we deal with Euclidian geometry close
to the real axis and pseudohyperbolic geometry far away from the real
axis (see $e.g.$ \cite[page 715]{Se98}). Let $\Lambda=\left\{ \lambda_{n}\right\} _{n\geq1}$
be a sequence of complex numbers. Let $N\geq1$ be an integer and
$\eta\in\left(0,\frac{1}{2}\right)$. For $\lambda\in\Lambda$, we
define 
\[
D_{\lambda,\eta}:=\left\{ z\in\mathbb{C}:\;\delta(\lambda,z)<\eta\right\} ,
\]
\[
N_{\lambda}:=\left\{ \mu_{\lambda,i}:\;1\leq i\leq N\right\} \subset\Lambda
\]
 as the set of $N$ closest neighboors of $\lambda$ (including in
particular$\lambda$) with respect to the distance $\delta$. Then
we set
\[
\sigma_{\lambda}:=D_{\lambda,\eta}\cap N_{\lambda},\qquad n_{\lambda}:=\left|\sigma_{\lambda}\right|\leq N.
\]
Note that the set $N_{\lambda}$, and consequently $\sigma_{\lambda}$,
is not unique. It is now natural to introduce the space (for $1<p<\infty$)
\[
X_{\tau}^{p}(\Lambda,N):=\left\{ a=\left(a(\lambda)\right)_{\lambda\in\Lambda}:\:\left\Vert a\right\Vert _{X_{\tau}^{p}(\Lambda,N)}<\infty\right\} ,
\]
where
\[
\left\Vert a\right\Vert _{X_{\tau}^{p}(\Lambda,N)}^{p}:=\sum_{\lambda\in\Lambda}\left(1+\left|\text{Im}(\lambda)\right|\right)\sum_{k=1}^{n_{\lambda}}\left|\tilde{\Delta}_{\sigma_{\lambda}}^{k-1}\left(ae^{\pm i\tau\cdot}\left(\mu^{(k)}\right)\right)\right|^{p}
\]
with 
\[
\tilde{\Delta}_{\sigma_{\lambda}}=\begin{cases}
\Delta_{\sigma_{\lambda}}, & \text{ if }\sigma_{\lambda}\cap\left\{ z\in\mathbb{C}:\:\left|\text{Im}(z)\right|<1\right\} =\emptyset\\
\square_{\sigma_{\lambda}}, & \text{ if not}
\end{cases}
\]
and
\[
e^{\pm i\tau\mu}=\left\{ \begin{array}{cl}
e^{i\tau\mu} & \text{, if }\mu\in\sigma_{\lambda}\text{and }\sigma_{\lambda}\cap\left\{ z\in\mathbb{C}:\text{Im}\left(z\right)\geq0\right\} \neq\emptyset\\
e^{-i\tau\mu} & \text{, otherwise}
\end{array}\right..
\]

\begin{rem}
\label{rem:espace independant de la partition}It can be shown that
if $\Lambda\cap\mathbb{C}_{a}^{\pm}$ is $N-$Carleson in the corresponding
half-plane, for each $a\in\mathbb{R}$, then this norm is equivalent
to the previously norm $\left\Vert \cdot\right\Vert _{X_{\tau,\epsilon}^{p}(\Lambda)}$
(for every $\epsilon>0$) defined in the above section. For the proof,
we refer to \cite[pp. 36-38]{Ha96a}. 
\end{rem}
The result is the following one.
\begin{thm}
\label{Thm : necessity}If $R_{\Lambda}$ is an isomorphism from $PW_{\tau}^{p}$
onto $X_{\tau}^{p}(\Lambda,N)$, then for every $a\in\mathbb{R}$,
$\Lambda\cap\mathbb{C}_{a}^{\pm}$ is $N'-$Carleson in the corresponding
half-plane, with $N'\leq N$.
\end{thm}
The proof is in two parts. We begin by showing that if $R_{\Lambda}$
is such an isomorphism, then $\Lambda_{a}^{\pm}$ is $N'-$Carleson
for some $N'\in\mathbb{N}$. This only requires the boundedness of
$R_{\Lambda}$. We first notice that by the Plancherel-Poly\`a theorem
(Proposition \ref{plancherel-polya}) the map
\[
\begin{array}{cccc}
\tau_{a}: & PW_{\tau}^{p} & \to & PW_{\tau}^{p}\\
 & f & \mapsto & f(\cdot+i\left(1+|a|\right)
\end{array}
\]
is an isomorphism and so $\tilde{R}_{\Lambda}:=R_{\Lambda}\circ\tau_{a}$
is still an isomorphism. Obviously, $\tilde{R}_{\Lambda}=R_{\tilde{\Lambda}}$,
where 
\[
\tilde{\Lambda}:=\Lambda+i\left(1+|a|\right).
\]
Note that for $\lambda\in\Lambda_{a}^{+}$, with the notations of
Lemma \ref{Interpolation Polynom}, 
\[
\left|a_{\lambda}\right|^{p}e^{-p\text{Im}(\lambda)}=\left|P_{\sigma_{\lambda},e^{\pm i\tau\cdot}a}\right|\le\sum_{k=1}^{n_{\lambda}}\left|\tilde{\Delta}_{\sigma_{\lambda}}^{k-1}\left(ae^{\pm i\tau\cdot}\left(\mu^{(k)}\right)\right)\right|^{p}
\]
and so $X_{\tau}^{p}\left(\tilde{\Lambda},N\right)$ injects into
$l^{p}\left(\left(1+\left|\text{Im}(\tilde{\lambda})\right|\right)e^{-p\left|\text{Im}(\tilde{\lambda})\right|}\right)$
so that 
\[
R_{\tilde{\Lambda}}:PW_{\tau}^{p}\to l^{p}\left(\left(1+\left|\text{Im}(\tilde{\lambda})\right|\right)e^{-p\left|\text{Im}(\tilde{\lambda})\right|}\right)
\]
is bounded. We set $\tilde{\Lambda}_{a}^{+}:=\Lambda_{a}^{+}+i\left(1+\left|a\right|\right)$
and reintroduce the inner function $I_{\tau}(z)=\exp\left(2i\tau z\right)$.
We have mentioned in the beginning of the paper that $PW_{\tau}^{p}$
is isomorphic to $K_{I^{\tau}}^{p}$, so

\[ R_{\tilde{\Lambda}_{a}^{+}}^{I_{\tau}}:= R_{\tilde{\Lambda}_{a}^{+}}\Big| K_{I_{\tau}}^{p}: K_{I_{\tau}}^{p}\to L^{p}\left(\mu_{\tilde{\Lambda}_{a}^{+}}\right) \]is
bounded, where
\[
\mu_{\tilde{\Lambda}_{a}^{+}}:=\sum_{\tilde{\lambda}\in\tilde{\Lambda}_{a}^{+}}\text{Im}(\tilde{\lambda})\delta_{\tilde{\lambda}}.
\]
In order to show that $\Lambda_{a}^{\pm}$ is $N'-$Carleson, it is
sufficient to show that $\mu_{\tilde{\Lambda}_{a}^{+}}$ is a Carleson
measure for $H_{+}^{p}$. Since in particular $\tilde{\Lambda}_{a}^{+}\subset\mathbb{C}_{1}^{+}$,
it is possible to find $\epsilon\in\left(0,1\right)$ such that
\[
\tilde{\Lambda}_{a}^{+}\subset L\left(I^{\tau},\epsilon\right):=\left\{ z\in\mathbb{C}^{+}:\:\left|I^{\tau}(z)\right|<\epsilon\right\} .
\]
Now, from a result of Treil and Volberg (see \cite{TV95} or \cite{Al97}),
the boundedness of $R_{\tilde{\Lambda}_{a}^{+}}^{I^{\tau}}$ implies
that
\begin{equation}
\sup_{I}\frac{\mu_{\tilde{\Lambda}_{a}^{+}}\left(\omega_{I}\right)}{m(I)}<\infty,\label{sup carleson window}
\end{equation}
where the supremum is taken over all the intervals of finite length
such that the Carleson window $\omega_{I}$ constructed on $I$ statisfies
\[
\omega_{I}\cap L(I^{\tau},\epsilon)\not=\emptyset.
\]
Observe that $L\left(I^{\tau},\epsilon\right)$ is in the upper half
plane $\mathbb{C}_{b}^{+}$, $b=\log\left(1/\epsilon\right)$, so
that if the length of the Carleson window is less than $b$, then
we have $\omega_{I}\cap L(I^{\tau},\epsilon)=\emptyset$. Hence, $\omega_{I}\cap\tilde{\Lambda}_{a}^{+}=\emptyset$
and so $\mu_{\tilde{\Lambda}_{a}^{+}}\left(\omega_{I}\right)=0$.
It follows that (\ref{sup carleson window}) is true for all finite
length intervals $I$, which is equivalent to the fact that $\mu_{\tilde{\Lambda}_{a}^{+}}$
is a Carleson measure or also that $\tilde{\Lambda}_{a}^{+}$ is $N'-$Carleson
and hence $\Lambda_{a}^{+}$ in the corresponding half-plane. Considering
the map
\[
\begin{array}{cccc}
s: & PW_{\tau}^{p} & \to & PW_{\tau}^{p}\\
 & f & \mapsto & f\left(-\cdot\right)
\end{array}
\]
which is also an isomorphism, we will also have the result for $\Lambda_{a}^{-}$.

Now, we want to prove that $N'\leq N$. In the following, if $\Lambda_{a}^{+}$
is $\left(N+k\right)-$Carleson, we write 
\[
\Lambda_{a}^{+}=\bigcup_{n\geq1}\tau_{n}^{k},
\]
where the groups $\tau_{n}^{k}$ come from the Generalized Carleson
condition, and so it is possible to assume that
\[
\text{diam}_{\delta}\left(\tau_{n}^{k}\right)<\frac{\eta}{4}
\]
(which in particular implies that $\tau_{n}^{k}\subset D_{\lambda,\eta}$)
and
\[
\gamma:=\inf_{n\not=m}\delta\left(\tau_{n}^{k},\tau_{m}^{k}\right)>0.
\]
We need the following lemma and its corollary. For technical reasons,
let us assume (without loss of generality) that $\Lambda_{a}^{+}\subset\mathbb{C}_{1}^{+}$
so that we can deal with the pseudohyperbolic metric and the corresponding
divided differences. 
\begin{lem}
If $R_{\Lambda}$ is an isomorphism from $PW_{\tau}^{p}$ onto $X_{\tau}^{p}(\Lambda,N)$
and $\Lambda_{a}^{+}$ is $\left(N+k+1\right)-$Carleson, $k\ge0$,
then it is possible to find $\vartheta>0$ such that every $\tau_{n}^{k+1}$
with $\left|\tau_{n}^{k+1}\right|=N+k+1$ satifies $\text{diam}_{\rho}\left(\tau_{n}^{k+1}\right)>\vartheta$.\end{lem}
\begin{proof}
Let us suppose to the contrary that we can find a subsequence $(\tilde{\tau}_{j})$
of $(\tau_{n}^{k+1})$ such that $|\tilde{\tau}_{j}|=N+k+1$ and $\text{diam}_{\rho}(\tilde{\tau}_{j})\to0$,
$j\to\infty$. We set $\tilde{\tau}_{j}=\{\lambda_{i}^{j}:\; i=0,..,N+k\}$.
Let us now introduce the sequence $a^{j}=(a^{j}(\lambda))_{\lambda\in\Lambda}$
defined by
\[
a^{j}(\lambda):=0,\;\lambda\not\not=\lambda_{N+k}^{j},
\]
and
\[
a^{j}(\lambda_{N+k}^{j}):=e^{\tau\text{Im}\left(\lambda_{N+k}^{j}\right)}\text{Im}\left(\lambda_{N+k}^{j}\right)^{-\frac{1}{p}}\frac{\underset{i\not=N+k}{\prod}\left|b_{\lambda_{i}^{j}}\left(\lambda_{N+k}^{j}\right)\right|}{\underset{i\not=N+k}{\max}\left|b_{\lambda_{i}^{j}}\left(\lambda_{N+k}^{j}\right)\right|}.
\]
Let 
\[
M_{j}:=\left\{ \lambda\in\Lambda_{a}^{+}:\:\lambda_{N+k}^{j}\in\sigma_{\lambda}\right\} 
\]
the set of the points of $\Lambda_{a}^{+}$ close to $\lambda_{N+k}^{j}$.
Since $\text{diam}_{\rho}\left(\tilde{\tau}_{j}\right)<\frac{\eta}{4}$
and $\lambda_{N+k}^{j}\in\tilde{\tau}_{j}$ we have for every $\lambda\in\tilde{\tau}_{j}$
that $\lambda_{N+k}^{j}\in\sigma_{\lambda}$, $i.e.$ $\tilde{\tau}_{j}\subset M_{j}$.
So, let $B_{j}:=M_{j}\setminus\tilde{\tau}_{j}$. Also, since $\Lambda_{a}^{+}$
is $\left(N+k+1\right)-$Carleson, 
\[
\sup_{j}\left|\Lambda_{a}^{+}\cap D_{\lambda_{N+k}^{j}}\right|<\infty,
\]
which implies that 
\[
\sup_{j}\left|M_{j}\right|<\infty.
\]
By construction, 
\[
\left\Vert a^{j}\right\Vert _{X_{\tau}^{p}(\Lambda,N)}^{p}=\sum_{\lambda\in M_{j}}\left(1+\text{Im}(\lambda)\right)\sum_{l=1}^{n_{\lambda}}\left|\Delta_{\sigma_{\lambda}}^{l-1}\left(a^{j}e^{i\tau\cdot}\left(\mu^{(l)}\right)\right)\right|^{p}.
\]
(Observe that we only consider in the sum the points containing $\lambda_{N+k}^{j}$
in their neighborhood.) 

We have to evaluate this expression. Take $\lambda\in M_{j}$. We
recall that $n_{\lambda}=\left|\sigma_{\lambda}\right|$. Note also
that for every $1\le l\le n_{\lambda}$, the divided difference 
\[
\left|\Delta_{\sigma_{\lambda}}^{l-1}\left(a^{j}e^{\pm i\tau\cdot}\left(\lambda^{(l)}\right)\right)\right|
\]
will be equal either to $0$ or to 
\[
\left|a^{j}\left(\lambda_{N+k}^{j}\right)e^{\pm i\tau\lambda_{N+k}^{j}}\prod_{m\in\omega_{l}}b_{\lambda_{m}^{j}}\left(\lambda_{N+k}^{j}\right)\right|,
\]
where $\omega_{l}\subset\sigma_{\lambda}$ contains $l-1$ points.
Now, $\omega_{l}= \omega_{l,1}\cup\omega_{l,2}$ where $\omega_{l,1} =\sigma_{\lambda}\cap \tilde{\tau}_j$
and $\omega_{l,2}$ are the other points. Note that $\omega_l$ cannot
contain $\lambda^j_{N+k}$. By assumption, for $\mu\in\omega_{l,2}$,
$|b_{\mu}(\lambda^j_{N+k})|\ge \gamma$. Hence,

\begin{eqnarray*}
\sum_{l=1}^{n_{\lambda}}\left|\Delta_{\sigma_{\lambda}}^{l-1}\left(a^{j}e^{\pm i\tau\cdot}\left(\lambda^{(l)}\right)\right)\right|^{p}\\
 &  & \negthickspace\negthickspace\negthickspace\negthickspace\negthickspace\negthickspace\negthickspace\negthickspace\negthickspace\negthickspace\negthickspace\negthickspace\negthickspace\negthickspace\negthickspace\negthickspace\negthickspace\negthickspace\negthickspace\negthickspace\negthickspace\negthickspace\negthickspace\leq\sum_{l=1}^{n_{\lambda}}\frac{\underset{i\neq N+k}{\prod}\left|b_{\lambda_{i}^{j}}\left(\lambda_{N+k}^{j}\right)\right|^{p}}{\underset{i\neq N+k}{\text{max}}\left|b_{\lambda_{i}^{j}}\left(\lambda_{N+k}^{j}\right)\right|^{p}}\cdot\frac{1}{\text{Im}\left(\lambda_{N+k}^{j}\right)\underset{\mu\in\omega_{l}}{\prod}\left|b_{\mu}\left(\lambda_{N+k}^{j}\right)\right|^{p}}\\
 &  & \negthickspace\negthickspace\negthickspace\negthickspace\negthickspace\negthickspace\negthickspace\negthickspace\negthickspace\negthickspace\negthickspace\negthickspace\negthickspace\negthickspace\negthickspace\negthickspace\negthickspace\negthickspace\negthickspace\negthickspace\negthickspace\negthickspace\negthickspace\leq\sum_{l=1}^{n_{\lambda}}\frac{1}{\gamma^{p\left|\omega_{l,2}\right|}}\frac{\underset{\xi\in\Omega_{l}}{\prod}\left|b_{\xi}\left(\lambda_{N+k}^{j}\right)\right|^{p}}{\underset{i\neq N+k}{\text{max}}\left|b_{\lambda_{i}^{j}}\left(\lambda_{N+k}^{j}\right)\right|^{p}}\cdot\frac{1}{\text{Im}\left(\lambda_{N+k}^{j}\right)}\\
 &  & \negthickspace\negthickspace\negthickspace\negthickspace\negthickspace\negthickspace\negthickspace\negthickspace\negthickspace\negthickspace\negthickspace\negthickspace\negthickspace\negthickspace\negthickspace\negthickspace\negthickspace\negthickspace\negthickspace\negthickspace\negthickspace\negthickspace\negthickspace\lesssim\frac{N}{\text{Im}\left(\lambda_{N+K}^{j}\right)},
\end{eqnarray*}
where $\Omega_{l}=\{\lambda_i^j:i=0,\ldots,N+k-1\}\setminus\omega_{l,1}$
are subsets of $\tilde{\tau}_{j}$. The last of the above inequalities
comes from the observation that $\Omega_l$ contains at least: 
\[
N+k-\left|\omega_{l,1}\right|\geq N+k-\left(n_{\lambda}-1\right)\geq N+k-\left(N-1\right)=k+1\geq1
\]
 points. We deduce that $a^{j}\in X_{\tau}^{p}(\Lambda,N)$ and that
its norm is uniformly bounded. Now, since $R_{\Lambda}$ is onto,
there is $f^{j}\in PW_{\tau}^{p}$ such that $f^{j}|\Lambda=a^{j}$
and 
\[
\left\Vert f^{j}\right\Vert _{PW_{\tau}^{p}}\lesssim\left\Vert a^{j}\right\Vert _{X_{\tau}^{p}(\Lambda,N)}\lesssim1.
\]
Setting $\tilde{f}^{j}:=e^{i\tau\cdot}f^{j}$, it follows from the
Plancherel-Poly\`a inequality that $\tilde{f}^{j}\in H_{+}^{p}$
and since $\Lambda_{a}^{+}$ is $\left(N+k+1\right)-$Carleson in
$\mathbb{C}^{+}$, Theorem (\ref{thm Hartmann}) implies in particular
that
\[
\text{Im}(\lambda_{N+K}^{j})\left|\Delta_{\tilde{\tau}_{j}}^{N+k}\left(\tilde{f}^{j}\left(\left(\lambda^{j}\right)^{\left(N+k+1\right)}\right)\right)\right|^{p}\lesssim\left\Vert f^{j}\right\Vert \lesssim1.
\]
But by construction, we have 
\[
\text{Im}(\lambda_{N+K}^{j})\left|\Delta_{\tilde{\tau}_{j}}^{N+k}\left(\tilde{f}^{j}\left(\left(\lambda^{j}\right)^{\left(N+k+1\right)}\right)\right)\right|^{p}=\frac{1}{\underset{i\not=N+k}{\max}\rho\left(\lambda_{N+k}^{j},\lambda_{i}^{j}\right)}
\]
which tends to $\infty$, $j\to\infty$ because $\text{diam}_{\rho}\tilde{\tau}_{j}$
tends to $0$, $j\to\infty$, which gives the required contradiction.
\end{proof}
The following corollary to the previous lemma allows us to end the
proof of our theorem.
\begin{cor}
If $R_{\Lambda}$ is an isomorphism from $PW_{\tau}^{p}$ onto $X_{\tau}^{p}(\Lambda,N)$
and $\Lambda_{a}^{+}$ is $\left(N+k+1\right)-$Carleson, $k\ge0$,
then $\Lambda_{a}^{+}$ is $\left(N+k\right)-$Carleson.\end{cor}
\begin{proof}
We write $\Lambda_{a}^{+}=\bigcup_{n\geq1}\tau_{n}^{k+1}$ with $|\tau_{n}^{k+1}|\leq N+k+1$.
Let us suppose that there are infinitely many $n$ for which we have
$|\tau_{n}^{k+1}|=N+k+1$ and let $Z$ be the set of such $n$. Because
of the previous lemma, we can find $\vartheta>0$ such that $\text{diam}_{\rho}(\tau_{n}^{k+1})>\vartheta$
for $n\in Z$. Then, for every $n\in Z$, it is possible to write
$\tau_{n}^{k+1}=\left\{ \lambda_{i}^{n}:\; i=1,\ldots,N+k+1\right\} $
such that 
\[
\rho\left(\lambda_{i},\lambda_{N+k+1}^{n}\right)\geq\frac{\vartheta}{2\left(N+k\right)},\quad i=1,\ldots,N+k.
\]
It follows that 
\[
\Lambda_{a}^{+}=\bigcup_{n\not\in Z}\tau_{n}^{k+1}\cup\left(\bigcup_{n\in Z}\tau_{n}^{k+1}\setminus\left\{ \lambda_{N+k+1}^{n}\right\} \right)\cup\left(\bigcup_{n\in Z}\left\{ \lambda_{n+k+1}^{n}\right\} \right)
\]
is a disjoint union of sets $\sigma_{n}$ with $\left|\sigma_{n}\right|\leq N+k$
and it can be shown that the sequence of Blascke products $\left(B_{\sigma_{n}}\right)_{n}$
satisfies the Generalized Carleson condition and hence that $\Lambda_{a}^{+}$
is $\left(N+k\right)-$Carleson.
\end{proof}

\thanks{I would like to thank Andreas Hartmann for his very helpful and permanent
support during this research and, more generally, from the beginning
of my thesis.}

\textsc{\small Equipe d'Analyse, Institut de Math\'ematiques de Bordeaux, Universit\'e Bordeaux 1, 351 cours de la Lib\'eration 33405 Talence C\'edex,
France.}{\small \par}

\emph{\bigskip}

\emph{E-mail address}: $\texttt{frederic.gaunard@math.u-bordeaux1.fr}$
\end{document}